\documentclass[10pt]{amsart}
\usepackage{amsxtra, amsfonts, amsmath, amsthm, amstext, amssymb, amscd, mathrsfs, verbatim, color}
\usepackage{threeparttable}
\usepackage{mathtools}
\usepackage[ansinew]{inputenc}\usepackage[T1]{fontenc}
\usepackage[all,cmtip]{xy}
\theoremstyle{plain}

\newtheorem{theorem}{Theorem}[section]
\newtheorem{lemma}[theorem]{Lemma}
\newtheorem{definition-theorem}[theorem]{Definition-Theorem}
\newtheorem{proposition}[theorem]{Proposition}
\newtheorem{corollary}[theorem]{Corollary}

\newtheorem{conjecture}[theorem]{Conjecture}

\theoremstyle{definition}
\newtheorem{definition}[theorem]{Definition}
\newtheorem{example}[theorem]{Example}
\newtheorem{remark}[theorem]{Remark}
\newtheorem{notation}[theorem]{Notation}
\newcommand \bth[1] { \begin{theorem}\label{t#1} }
\newcommand \ble[1] { \begin{lemma}\label{l#1} }

\newcommand \bpr[1] { \begin{proposition}\label{p#1} }
\newcommand \bco[1] { \begin{corollary}\label{c#1} }
\newcommand \bde[1] { \begin{definition}\label{d#1}\rm }
\newcommand \bex[1] { \begin{example}\label{e#1}\rm }
\newcommand \bre[1] { \begin{remark}\label{r#1}\rm }

\newcommand \bnota[1] {\begin{notation}\label{n#1}\rm }
\newcommand {\ele} { \end{lemma} }

\newcommand {\epr} { \end{proposition} }
\newcommand {\eco} { \end{corollary} }
\newcommand {\ede} { \end{definition} }
\newcommand {\eex} { \end{example} }
\newcommand {\ere} { \end{remark} }
\newcommand {\enota} { \end{notation} }

\DeclareMathOperator \Hom { {\mathrm{Hom}} }

\DeclareMathOperator \ind{ {\mathrm{ind}}}
\DeclareMathOperator \Ind { {\mathrm{Ind}} }

\DeclareMathOperator \sgn { { \mathrm{sgn}}}

\begin{document}
\setlength{\baselineskip}{1.2\baselineskip}
\title[Bernstein-Zelevinsky derivatives and filtrations]
{Bernstein-Zelevinsky derivatives, branching rules and Hecke algebras}

\author[Kei Yuen Chan]{Kei Yuen Chan}
\address{ Korteweg-de Vries Institute for Mathematics, Universiteit van Amsterdam}
\email{K.Y.Chan@uva.nl}

\author[Gordan Savin]{Gordan Savin}
\address{
Department of Mathematics \\
University of Utah}
\email{savin@math.utah.edu}

\begin{abstract}
Let $G$ be a split reductive group over a $p$-adic field $F$.  Let $B$ be a Borel subgroup and $U$ the maximal unipotent subgroup of $B$. 
 Let $\psi$ be a Whittaker character of $U$. Let $I$ be an Iwahori subgroup of $G$.  We describe the Iwahori-Hecke algebra action on the Gelfand-Graev representation 
$(\mathrm{ind}_{U}^{G}\psi)^I$ by an explicit projective module.  As a consequence, for $G=GL(n,F)$, we define and describe Bernstein-Zelevinsky derivatives of representations 
generated  by $I$-fixed vectors in terms of the corresponding Iwahori-Hecke algebra modules. 
Furthermore, using Lusztig's reductions, we show that the Bernstein-Zelevinsky derivatives can be determined using graded Hecke algebras. 

We give two applications of our study. Firstly, we compute the Bernstein-Zelevinsky derivatives of generalized Speh modules, which recovers a result of Lapid-M\'inguez and Tadi\'c. Secondly, we give a realization of the Iwahori-Hecke algebra action on some generic representations of $GL(n+1,F)$, restricted to $GL(n,F)$, which is further used to verify a conjecture on an Ext-branching problem of D. Prasad for a class of examples.

\end{abstract}
\maketitle 

\section{Introduction}

\subsection{} 

Bernstein-Zelevinsky derivatives were first introduced and studied in \cite{BZ} and \cite{Ze} and are important for the classification of simple representations of $GL(n,F)$. The derivatives have other applications in representation theory such as branching rules \cite{Pr} and study of $L$-functions. 

One goal of this paper is to formulate a functor for Hecke algebras that corresponds to the Bernstein-Zelevinsky derivative and show that the Bernstein-Zelevinsky derivatives can be determined from the corresponding functor. The functor thus provides a framework to understand some problems from the Hecke algebra approach. As an application of our study, we  compute the Bernstein-Zelevinsky derivatives of generalized Speh modules, which does not use the determinantal formula of Tadi\'c \cite{Ta} and Lapid-M\'inguez \cite{LM} or Kazhdan-Lusztig polynomials \cite{Ze2, CG}. 

Another consequence of our study attempts to understand the branching problem for the pair $(GL(n+1,F), GL(n,F))$. The $\mathrm{Hom}$-branching problem has been studied extensively, see for example \cite{Pr, GP, Pr3, GGP, AGRS}. The $\mathrm{Ext}$-branching problems were first initiated and studied by Dipendra Prasad \cite{Pr2}. Another result in this paper is to give a description of the localized Hecke algebra action on some generic representations of $GL(n+1,F)$, considered as representations of $GL(n,F)$, which is used to verify a conjecture of Prasad on Ext-multiplicity for some cases including all spherical generic representations of $GL(n,F)$. 

\subsection{Main results}  
Let $F$ be a $p$-adic field with the residual field of order $q$. 
Let $G$ be a split reductive reductive group over $F$.  Let $T$ be a maximal split torus in $G$ and let $W$ be the Weyl group. 
 We fix a Chevalley-Steinberg pinning of $G$, and emphasize that the data introduced here depends on the choice of the pinning.  Precise definitions are in Section \ref{s induced rep wc}. 
 Let $B=TU$ be a Borel subgroup with the maximal unipotent subgroup $U$ and $I$ be an Iwahori subgroup of $G$.
  The Iwahori-Hecke algebra $\mathcal H$  is the convolution algebra of $I$-bi-invariant compactly supported functions on $G$. It contains a finite subalgebra $\mathcal H_W$ 
   of functions supported on the hyperspecial maximal compact subgroup determined by the pinning.  
  As the notation indicates, $\mathcal H_W$ has a basis $T_w$ of characteristic functions of double cosets parameterized by the Weyl group. The algebra 
$\mathcal H_W$ has a one dimensional representation $\sgn$, $T_w\mapsto (-1)^{l(w)}$,  where $l$ is a length function on $W$. A prominent role in this paper is played by the element
   \[ 
   \mathbf S= \sum_W (-1/q)^{l(w)} T_w \in \mathcal H_W.  
   \] 
   If $\sigma$ is an $\mathcal H_W$-module, then $\mathbf S(\sigma)$ is the $\sgn$-isotypic subspace of $\sigma$.  We shall informally call $\mathbf S$ a sign projector.

Let $\psi$ be a Whittaker character of $U$.
Perhaps the most important result in this paper is a description of the space $\mathrm{ind}_U^G \psi$ in terms of Hecke algebra actions:

\begin{theorem} (Corollary \ref{cor realize iwahoir fix vector}) \label{thm intro realize}
As $\mathcal H$-modules, $(\mathrm{ind}_{U}^{G}\psi)^{I} $ is isomorphic to $\mathcal H \otimes_{\mathcal H_{W}} \mathrm{sgn}$.
\end{theorem}

Bushnell and Henniart \cite{BH} have studied Bernstein components of $\mathrm{ind}_U^G \psi$  and have shown, among other things, that each component
is a finitely generated $G$-module. Our result is therefore a refinement of theirs, for the particular component. 
The use of $\mathcal H \otimes_{\mathcal H_{W}} \mathrm{sgn}$ is independently inspired from the study of \cite{Ch} and \cite{Sa}.
We remark that the occurrence of $\mathrm{sgn}$ for representations admitting Whittaker models appeared in the study of Barbasch-Moy \cite{BM}. 
Our  Corollary \ref{cor recover results} strengthens their result to the category of smooth representations.

Theorem \ref{thm intro realize} plays an important role in the formulation of the Bernstein-Zelevinsky derivatives in the language of Hecke algebras. Let $G_n =GL(n,F)$ and 
$\pi$ a smooth representation of $G_n$. 
The $i$-th Bernstein-Zelevinsky derivative of $\pi$ is a $G_{n-i}$-representation, denoted $\pi^{(i)}$,  
obtained by applying a twisted Jacquet functor on $\pi$, in which the Whittaker character is involved (see Section \ref{ss ZD} for the detailed formulation).

Let $I_n$ denote the Iwahori subgroup of $G_n$ and $\mathcal H_n$ the Iwahori-Hecke algebra. 
The Weyl group of $G_n$ is isomorphic to the group $S_n$ of all permutation matrices.  Let 
$\mathbf S_n \in\mathcal H_{S_n}$ be the sign projector.  For every $i=1, \ldots, n-1$,  $\mathcal H_{n-i} \otimes \mathcal H_i$ is the Iwahori-Hecke algebra of 
a Levi subgroup of $G_n$. Using Bernstein's generators and relations $\mathcal H_{n-i} \otimes \mathcal H_i$ can be viewed as a subalgebra of $ \mathcal H_n$.  In particular, 
the map $h\mapsto h\otimes 1$ realizes $\mathcal H_{n-i}$ as a subalgebra of $\mathcal H_{n}$. Let $\mathbf S_i^n$ be the image in 
$\mathcal H_n$ of $1\otimes \mathbf S_i$, where $\mathbf S_i$ is the sign projector in $\mathcal H_{i}$.  For every $\mathcal H_n$-module $\sigma$, 

\begin{align} \label{eqn intro ZD Hecke} 
\mathbf{BZ}_i(\sigma):=\mathbf S_i^n (\sigma). 
\end{align}
is naturally an $\mathcal H_{n-i}$-module.

\begin{theorem} (Theorem \ref{thm bz aha}) \label{thm intro bz aha}
Let $\pi$ be a smooth representation of $G_n$. 
 Let  $\mathbf{BZ}_i$ be the functor defined in (\ref{eqn intro ZD Hecke}). There is a natural isomorphism of $\mathcal H_{n-i}$-modules
\[     (\pi^{(i)})^{I_{n-i}} \cong\mathbf{BZ}_i(\pi^{I_n}). 
\]
\end{theorem}

One then can similarly formulate the Bernstein-Zelevinsky derivative for graded Hecke algebras. We check in Section \ref{s bz gha} that Bernstein-Zelevinsky derivatives between affine Hecke algebras and graded Hecke algebras agree under the Lusztig's reductions. A reason for formulating the Bernstein-Zelevinsky derivatives for the graded Hecke algebra is that the theory of the symmetric group is relatively easier to apply. In particular, we use the Littlewood-Richardson rule for computing the Bernstein-Zelevinsky derivatives of 
 generalized Speh representations. For the detailed notations, one refers to Section \ref{s bz speh}.

\begin{corollary} (Corollary \ref{cor bz speh})
Let $\pi$ be a generalized Speh representation of $GL(n,F)$ associated to a partition $\bar{n}$ of $n$. Then the $i$-th Bernstein-Zelevinsky derivatives $\pi^{(i)}$ is the direct sum of generalized Speh representations corresponding to the partitions obtained by removing $i$ boxes from $\bar{n}$ but at most one in each row, such that the resulting diagram is still a Young diagram.

\end{corollary}
The generalized Speh modules correspond to the single $S_n$-type Hecke algebra modules studied by Barbasch-Moy \cite{BM3} and Ciubotaru-Moy \cite{CM}. 
Because of the simple type structure,  their Bernstein-Zelevinsky derivatives can be computed from the theory of symmetric groups.

We remark that Corollary \ref{cor bz speh} is independently proved by Lapid-M\'inguez \cite{LM} (following a suggestion of Tadi\'c) and their result also covers a larger class which they call ladder representations.

We now turn to another direction of our study on branching problems for the pair $(GL(n+1, F),GL(n,F))$. A useful tool in studying that problem is the Bernstein-Zelevinsky geometric lemma. More precisely, the geometric lemma says that a smooth representation $\pi$ of $GL(n+1,F)$ restricted to the mirabolic subgroup $E_n$ admits a finite $E_n$-filtration such that the successive quotients can be described in terms of certain induction functors and twisted Jacquet functors (see Theorem \ref{thm bz filtration} for the details). We shall call those successive quotients to be the Bernstein-Zelevinsky composition factors. Whittaker characters and Bernstein-Zelevinsky derivatives are involved in defining the functors and hence, in principle, Theorem \ref{thm intro realize} and Theorem \ref{thm intro bz aha} can be applied to study the Bernstein-Zelevinsky composition factors. 

When restricting $\pi$ from $GL(n+1, F)$ to $GL(n,F)$ we shall only consider the Bernstein component (for $GL(n,F)$) of $\pi$ generated by the Iwahori-fixed vectors. Hence the formulation of our results necessitates additional notation involving the Iwahori-Hecke algebra $\mathcal H_n$. Let $\mathcal Z_n$ be the center of $\mathcal H_n$ and let $\mathcal J$ be a maximal ideal in $\mathcal Z_n$. Abusing language, representations annihilated by $\mathcal J$ will be said to have the central character $\mathcal J$. 
Let $\widehat{\mathcal Z}_n$ be  the $\mathcal J$-adic completion of ${\mathcal Z}_n$. 
To study the Bernstein-Zelevinsky composition factors, it is easier to deal with their $\mathcal J$-adic completions. 
Let $\widehat{\mathcal H}_n= \widehat{\mathcal Z}_n \otimes_{\mathcal Z}\mathcal H_n$.
 The $\mathcal J$-adic completion of an $\mathcal H_n$-module $\chi$ is  
   the $\widehat{\mathcal H}_n$-module  $\widehat{\chi}=\widehat{\mathcal Z}_n \otimes_{\mathcal Z_n} \chi$. 
 For a finite-dimensional $\mathcal H_n$-module $\chi$, the $\mathcal J$-adic completion  is simply the summand of  $\chi$ annihilated by a power of $\mathcal J$.


It is hard to compute the $\mathcal J$-adic completion of $\pi$ in general. However, some classes of examples of $\pi$, which we call locally nice representations at $\mathcal  J$, have a simple description of the completion. The structure will be explained in Theorem \ref{thm intro local structure} below. See Definition \ref{def relative generic} and Example \ref{ex relatively generic} for the term locally nice.  We know some immediate examples. For example, if there is unique isomorphism class of irreducible representations annihilated by $\mathcal J$,   then any generic representation $\pi$ of $GL(n+1,F)$ is locally nice at $\mathcal J$. 
As another extreme, the Steinberg representation of $GL(n+1,F)$ is locally nice at every central character of $\mathcal H_n$ (see Theorem \ref{thm global st} and Corollary \ref{cor st locally nice}).


Now we state another consequence of our study:

\begin{theorem} \label{thm intro local structure} (Theorem \ref{thm local structure})
Let $\pi$ be an irreducible generic representation of $GL(n+1,F)$ and let $I_n$ be the Iwahori subgroup of $GL(n,F)$. Regard $(\pi|_{GL(n,F)})^{I_n}$ as an $\mathcal H_n$-module.  
  Let $ \mathcal J$  be a maximal ideal in  ${\mathcal Z}_n$.  
  Suppose $\pi$ is locally nice at $\mathcal J$ (see Definition \ref{def relative generic} and Example \ref{ex relatively generic}). 
Then  the $\mathcal J$-adic completion of $(\pi|_{GL(n,F)})^{I_n}$ is isomorphic to $\widehat{\mathcal H}_n \otimes_{\mathcal H_{S_n}} \mathrm{sgn}$ and hence is projective in the category of $\widehat{\mathcal H}_n$-modules.
\end{theorem}
For some comments on the proof of Theorem \ref{thm intro local structure}, see the paragraphs before Theorem \ref{thm local structure}.

One may think that locally nice representations have the simplest local structure. The complication of the local structure of a restricted generic representation starts to increase outside this class, and hence deeper understanding of the structure is needed. Also, determining the central characters at which a generic representation is locally nice is an interesting problem.

As a consequence, we obtain sufficient structural information to verify a conjecture of D. Prasad for those locally nice representations. We first recall the conjecture:

\begin{conjecture}[Prasad] \cite[Conjecture 1]{Pr2}\label{conj prasad}
Let $\pi_1$ be an irreducible generic representation of $GL(n+1,F)$ and let $\pi_2$ be an irreducible generic representation of $GL(n,F)$. Then 
\[   \mathrm{Ext}^i_{GL_{n}(F)}(\pi_1, \pi_2) =0
\]
for all $i \geq 1$. (Here $\mathrm{Ext}^i_{GL(n,F)}$ is taken in the category of smooth representations of $GL(n,F)$.)
\end{conjecture}

\begin{corollary} (Corollary \ref{cor ext branching}) \label{cor intro branch}
Let $\pi_2$ be an irreducible generic representation of $GL(n,F)$ with Iwahori fixed vectors with the central character $\mathcal J$. 
Suppose $\pi_1$ is an irreducible representation of $GL(n+1,F)$ which is locally nice at $\mathcal J$. Then 
\[   \mathrm{Ext}^i_{GL(n,F)}(\pi_1, \pi_2) =0
\]
for all $i \geq 1$. 
\end{corollary}
The cases we considered in Corollary \ref{cor intro branch} can be seen as the simplest ones in the sense of Theorem \ref{thm intro local structure}, but still cover some cases that cannot be merely deduced from the Bernstein-Zelevinsky composition factors using the Frobenius reciprocity,  central character considerations  and the Euler-Poincar\'e pairing.

\subsection{} We give some comments on other Bernstein components. We expect that results in this paper hold for other Bernstein components with a suitable reformulation with the use the theory of types by Bushnell-Kutzko \cite{BK, BK2}. However, our approach in Section \ref{s induced rep wc} cannot be adapted directly to other Bernstein components. 

\subsection{Acknowledgements} A part of this work was done during the Sphericity 2016 Conference and Workshop. The authors would like to thank the organizers for providing the excellent environment for discussions. The first author was supported by the Croucher Postdoctoral Fellowship. The second author was supported in part by NSF grant DMS-1359774.



\section{Iwahori-fixed vectors for the Gelfand-Graev representation} \label{s induced rep wc}

Let $G$ be a Chevalley group over a $p$-adic field $F$. Let $\mathcal O$ be the ring of integers of $F$, let $\varpi$ be the uniformizer of $F$ and let $\mathfrak{p}$ be the maximal ideal in $\mathcal O$. Let $q=\mathrm{card}(\mathcal O/\mathfrak{p})$. Let $B=TU$ be a Borel subgroup with a maximal unipotent subgroup $U$ and a torus $T$. The torus 
$T$ determines a root system $R$ and $U$ a set of simple roots $\Pi$  and positive roots $R^+$ 
for $R$. Let $W=N_G(T)/T$, where $N_G(T)$ is the normalizer of $T$ in $G$. 
We fix a Chevalley-Steinberg pinning of $G$. In particular, for every  $\alpha \in R$, we have a one-parameter subgroup in $G$ whose elements are denoted by 
$x_{\alpha}(t)$, where $t\in F$.  The group $U$ is generated by $x_{\alpha}(t)$ for $\alpha\in R^+$. 
 For $\alpha \in R$, let $w_{\alpha}(t)=x_{\alpha}(t)x_{-\alpha}(-t^{-1})x_{\alpha}(t)$. 
 We let $\dot{s}_{\alpha}=w_{\alpha}(1)$, where $s_{\alpha}$ is a reflection associated to $\alpha\in \Pi$. For a choice of reduced expression of
  $w=s_{\alpha_1}\ldots s_{\alpha_r}\in W$, we let $\dot{w}=w_{\alpha_1}(1)\ldots w_{\alpha_r}(1)$. It is a representative of $w$ and,  for $\alpha \in R$, 
  $\dot{w} x_{\alpha}(t) \dot{w}^{-1}=x_{w(\alpha)}(ct)$ for some $c \in \mathcal O^{\times}$.

Let $P$ be a closed subgroup of $G$. Let $(\pi, X)$ be a smooth representation of $P$. Denote by $\mathrm{Ind}_P^G \pi$ the normalized induction.  Denote by $\mathrm{ind}_P^G \pi$ the normalized compact induction. Denote by $\widetilde{\pi}$ the smooth dual of $\pi$.

If $P=MN$ is a parabolic subgroup with the Levi subgroup $M$ and the unipotent radical $N$, denote by $\pi_N$ the normalized Jacquet module of $\pi$.

Let $\overline{\psi}$ be an additive character of $F$ with conductor $\mathfrak{p}$. Fix a Whittaker character $\psi$ of $U$ such that 
\[ \psi\left(\Pi_{\alpha \in R^+} x_{\alpha}(t_{\alpha}) \right) = \overline{\psi} \left(\sum_{\alpha \in \Pi} t_{\alpha} \right) . \] 

Let $V=\mathrm{ind}_U^G \psi$. It is  the space of smooth functions $f$ on $G$ satisfying 
\begin{enumerate}
\item[(1)] $f$ is compactly supported modulo $U$, and 
\item[(2)] $f(ug)=\psi(u)f(g)$ for all $g \in G$, and $u\in U$. 
\end{enumerate}

 Let $\bar B=T\bar U$ be the Borel subgroup opposite to $B$, i.e. $\bar U$ is generated by $x_{\alpha}(t)$ for all $\alpha\in R^-$. 
Let $l: W \rightarrow \mathbb{Z}$ be the length function on $W$.  
Let $V_r$ be the subspace of $V$ consisting of all functions in $V$ supported in the union of cells $X_w=U w T\bar U$ for all $w\in W$  
such that $l(w) \leq r$. Note that each $V_r$ is a $\bar B$-submodule of $V$. 

\begin{lemma} \label{lem iso whittaker} The inclusion $V_0 \subseteq V$ induces an isomorphism of $T$-modules 
$(V_0)_{\bar U} \cong V_{\bar U}$.
\end{lemma}
\begin{proof}
For every $w\in W$, let $V_w$ be the space of smooth functions $f$ on $X_w$ such that $f(ux)=\psi(u) f(x)$ for all $u\in U$ and $x\in X_w$, and such that the support of $f$ 
is contained in $U wT_f  \bar U_f$ where $T_f$ is a compact subset of $T$ and $\bar U_f$ a compact subset of $\bar U$, both depending on $f$. 
For $r \geq 1$ we have an exact sequence 
\[ 
0 \rightarrow V_{r-1} \rightarrow V_r \rightarrow \bigoplus_{l(w)=r} V_w
\] 
obtained by restricting functions $f\in V_r$ to $X_w$ for $l(w)=r$. 
Each $V_w$ is an $\bar U$-module under the action by right translations. For $\bar u\in  \bar U$, let $R(\bar u)$ denote the right translation action. 
 
 Claim: $(V_w)_{U}=0$, if $l(w)>0$.  Proof: If $l(w)>0$, then there exists an open compact subgroup $\bar U_c$ of $\bar U$ such that 
\[ 
\int_{U \cap w \bar U_c {w}^{-1}} \psi (u)~ du=0. 
\] 
Let $f\in V_w$, and assume that $f$ is supported in $U\dot{w} T_f  \bar U_f$ where $T_f$ is a compact subset of $T$ and $\bar U_f$ a compact subset of $\bar U$. 
We can enlarge $\bar U_f$ so that it is a subgroup of $\bar U$ and, for every $t\in T_f$, $t  \bar U_f t^{-1}$ contains $\bar U_c$. 
It is a simple check that  
\[ 
\int_{\bar U_f} R(\bar u) (f) ~d\bar u =0. 
\] 
This proves the claim. 

By the exactness of the Jacquet functor, the claim implies that the inclusion $V_0\subset V$  of $\bar B$-modules gives an isomorphism 
$(V_0)_{\bar U}\cong V_{ \bar U}$ of $T$-modules. 
\end{proof}

\begin{proposition} (also see \cite[Theorem 1]{Sa}) \label{prop savin realize}
There exists an isomorphism of $T$-modules 
\[ \Phi: V_{U} \rightarrow C_{c}^{\infty}(T) 
\]
\end{proposition}
\begin{proof}
By Lemma \ref{lem iso whittaker}, it suffices to construct an isomorphism of $T$-modules between $(V_0)_{\bar U}$ and $C_{c}^{\infty}(T)$. An element in 
$V_0$ is a function supported on the open cell $UT\bar U$ and the restriction to $T\bar U$ gives a bijection between $V_0$ and compactly supported 
functions on $T\bar U$.  Fix an invariant measure on $\bar U$ such that the measure of $\bar U \cap I$ is 1. (This is a natural normalization coming from the 
pinning.) It is easy to check that the map from $V_0$ to $C_{c}^{\infty}(T)$ defined by 
\[  
f\mapsto  f_{U}(t) =\int_U f(t\bar u)~d\bar u 
\]
descends to an isomorphism of $(V_0)_{\bar U}$ and $C_{c}^{\infty}(T)$.  This gives $\Phi$.  
\end{proof}


\subsection{Iwahori-Hecke algebra action} \label{ss IHA} 
The choice of Chevalley-Steinberg pinning gives a structure to $G$ of a group scheme over $\mathcal O$ such that $G(\mathcal O)$ is a hyperspecial maximal 
compact subgroup. 
Let $I$ be the Iwahori subgroup of $G$ which is the inverse image of $\bar B(\mathcal O/\mathfrak p)$ under the map $G(\mathcal O) \rightarrow  G(\mathcal O/\mathfrak{p})$. 
 Let $\mathcal H=C_c(I \setminus G/ I)$ be the convolution algebra of compactly supported $I$-bi-invariant functions on $G$. 
 The double cosets are parameterized by an extended affine Weyl group $W_{\mathrm{ex}}=N_G(T)/T(\mathcal O)$. For $w \in W_{\mathrm{ex}}$, let $T_w$ be the characteristic function of the double coset $I{w}I$. We shall normalize the measure on $G$ such that $T_1$ is an identity element, equivalently, the volume $\mathrm{vol}(I)=1$.

Recall that  $q=\mathrm{card}(\mathcal O/\mathfrak{p})$. Define the length function $l: W_{\mathrm{ex}} \rightarrow \mathbb{Z}$ such that
\[  q^{l(w)}=[I{w}I: I]=[I:(I \cap {w}^{-1}I{w})] 
\]
Then we have $T_{w_1}T_{w_2}=T_{w_1w_2}$ if $l(w_1w_2)=l(w_1)+l(w_2)$ and $(T_s-q)(T_s+1)=0$ for $l(s)=1$. 

Let $X=\Hom(\mathbb G_m, T)$ be the co-character lattice.  
Then $T\cong X\otimes_{\mathbb Z} F^{\times}$, and $X$ can be considered a subgroup of $T$ by the homomorphism $x\mapsto \dot{x}=x\otimes \varpi^{-1}$. (Note the inverse!) 
 This homomorphism gives a bijection $X\cong T/T(\mathcal O)$.  It extends to an isomorphism between a semi-direct product of $X$ and $W$ and $W_{\mathrm{ex}}$ 
by mapping $w\in W$ to its representative $\dot{w}\in N_G(T)$ defined earlier.  Let $\langle \cdot, \cdot\rangle$ be the natural pairing between the co-character and character lattices. 
Let 
\[ 
X_{\mathrm{dom}}= \left\{ x \in X: \langle x, \alpha \rangle \geq 0 \right\} .
\]
Any element $x \in X$ can be written as a linear combination as $x=y-z$ for $y,z \in X_{\mathrm{dom}}$. 
Following from Bernstein, let $\theta_x=q^{-(l(y)-l(z))/2}T_yT_z^{-1}$. 
Let $\mathcal A$ be the commutative subalgebra of $\mathcal H$ generated by $\theta_x$ for $x \in X$. 
The algebra $\mathcal A$ is isomorphic to the group algebra 
$\mathbb C[X]$, by the isomorphism $x\mapsto \theta_x$. 

For a smooth representation $(\pi, E)$ of $G$, denote by $E^I$ or, abusing notation by $\pi^I$ if the vector space $E$ is not specified, 
 the subspace of $I$-fixed vectors of $\pi$. 
The space $\pi^I$ is equipped with a $\mathcal H$-module structure by convolution. 

Let $I_T=I \cap T=T(\mathcal O)$. For any $T$-module, the subspace of $I_T$-fixed vectors is a module for $T/T(\mathcal O)\cong X$. Thus, it is a $\mathbb C[X]$-module. 
We have the following theorem, due to Borel, Casselman, Matsumoto and Bernstein \cite{Bo}: 

\begin{theorem} \label{thm BorelMat} 
Let $(\pi, E)$ be a smooth $G$-module. As $\mathcal A\cong \mathbb C[X]$-modules
\[   E^{I} \cong E_{\bar U}^{I_T} 
\]
The isomorphism map is defined from the natural map from $E$ to $E_{\bar U}$.
\end{theorem}

We shall apply this result to $V=\ind_U^G(\psi)$. 
By Proposition \ref{prop savin realize}, we  $V_{\bar U} \cong C^{\infty}_c(T)$.  Note that $C^{\infty}_c(T)^{I_T}\cong C_c(T/I_T)\cong \mathbb C[X]$. 
Let $\mathrm{ch}_{I_T}\in C^{\infty}_c(T)$ be the characteristic function of $I_T$. Under the isomorphism $C^{\infty}_c(T)^{I_T}\cong \mathbb C[X]$, 
the function $\mathrm{ch}_{I_T}$ corresponds to $1\in \mathbb C[X]$. Thus it is a generator of this $\mathbb C[X]$-module. We shall now describe a corresponding generator in 
$V^I$ in the following lemma.

\begin{lemma} \label{lem computations on ch}
 Let $\mathrm{ch}_{I}^{\psi}$ be a function on $G$, supported on $U\cdot (I\cap \bar B)$ such that $\mathrm{ch}_{I}^{\psi}(u i)=\psi(u)$ for all 
 $u\in U$ and $i\in \bar B \cap I$.  
Then 
\begin{enumerate}
\item $\mathrm{ch}_I^{\psi} \in V_0$,
\item $\mathrm{ch}_I^{\psi} \in V^I$,
\item $T_w \cdot \mathrm{ch}_I^{\psi}=(-1)^{l(w)}\mathrm{ch}_I^{\psi}$, for $w \in W$, and
\item $\Psi(\mathrm{ch}_I^{\psi})=\mathrm{ch}_{I_T}$, where $\Psi$ is the isomorphism of $V^I$ and $V_{\bar U}^{I_T}$. 
\end{enumerate}
\end{lemma}

\begin{proof}
(1) is obvious. (2) follows from the decomposition $I=(I\cap U) \cdot (I\cap \bar B)$ and the fact that $\psi$ is trivial on $I\cap U$. 
For (3) it suffices to check the equation for $T_{s_{\alpha}}$, where $s_{\alpha}$ is the reflection corresponding to a simple root $\alpha$. Using the decomposition 
$G=U W_{\mathrm{ex}} I$ (see \cite{HKP}) we need to compute $T_{s_{\alpha}} \cdot \mathrm{ch}_I^{\psi}(w)$ for every $w\in W_{\mathrm{ex}}$: 
\[
 T_{s_{\alpha}}\cdot \mathrm{ch}_{I}^{\psi}(w) = \int_{g \in IsI} \mathrm{ch}_{I}^{\psi}(wg) dg = \sum_{t \in \mathcal O/\mathfrak{p}} \mathrm{ch}^{\psi}_{I}(wx_{-\alpha}(t) w_{\alpha}(1)). 
\]
Let $\mathbf w$ be the projection of $w$ in $W$. We need the following version of Bruhat lemma, recall that $\alpha$ is a simple root: 
\[ 
UwIs_{\alpha}I=
\begin{cases} 
Uws_{\alpha}I \text{ if } \mathbf w(\alpha)<0 \text{ and } \\
Uws_{\alpha}I \cup UwI  \text{ if } \mathbf w(\alpha)>0. \\
\end{cases} 
\] 
Hence $T_{s_{\alpha}}\cdot \mathrm{ch}_{I}^{\psi}(w) =0$ if $w\neq s_{\alpha},1$. Assume now that $w=s_{\alpha}$, and represent it by $w_{\alpha}(-1)=w_{\alpha}(1)^{-1}$. Then 
\[ 
\sum_{t \in O/\mathfrak{p}} \mathrm{ch}^{\psi}_{I}(w_{\alpha}(-1)x_{-\alpha}(t) w_{\alpha}(1))= 
\sum_{t \in O/\mathfrak{p}} \mathrm{ch}^{\psi}_{I}(x_{\alpha}(-t))=  \sum_{t \in O/\mathfrak{p}}\psi(t)=0. 
\] 
If $w=1$, then $\mathrm{ch}^{\psi}_{I}(x_{-\alpha}(t) w_{\alpha}(1))=0$ unless $t\in \mathcal O^{\times}$. If $t\in \mathcal O^{\times}$ then the relation 
\[
x_{-\alpha}(t) w_{\alpha}(1)= \begin{pmatrix} 1 & 0 \\ t & 1 \end{pmatrix} \begin{pmatrix} 0 & 1 \\ -1 & 0 \end{pmatrix} =\begin{pmatrix} 1 & t^{-1} \\ 0 & 1 \end{pmatrix} \begin{pmatrix} t^{-1} & 0 \\ -1 & t \end{pmatrix} \equiv x_{\alpha} (t^{-1}) \pmod{I}\]
and the invariance properties of $\mathrm{ch}^{\psi}_{I}$ give 
\[ 
\mathrm{ch}^{\psi}_{I}(x_{-\alpha}(t) w_{\alpha}(1))=\psi(t^{-1}). 
\] 
Summing up over $t\in (\mathcal O/\mathfrak{p})^{\times}$ yields $-1$. This completes (3). 
(4) is trivial.  
\end{proof}

Let $\mathcal H_W$ be the finite subalgebra of $\mathcal H$ generated by $T_w$ for $w \in W$.  Let 
$\sgn$ denote the one-dimensional representation of $\mathcal H_W$ on $\mathbb C$ where $T_w$ acts by $(-1)^{l(w)}$. Let $\pi$ be a smooth representation of 
$G$, so $\pi^I$ is an $\mathcal H$-module. 
We have the following, tautological, Frobenius reciprocity
\[ 
\Hom_{\mathcal H}(\mathcal H \otimes_{\mathcal H_{W}} \mathrm{sgn}, \pi^I)=\Hom_{\mathcal H_{W}}(\sgn, \pi^I), 
\] 
where an element $A'\in \Hom_{\mathcal H_{W}}(\sgn, \pi^I)$ corresponds to $A\in \Hom_{\mathcal H}(\mathcal H \otimes_{\mathcal H_{W}} \mathrm{sgn}, \pi^I)$ defined by 
$A(h\otimes 1)= \pi(h)(A'(1))$, for all $h\in H$. 

\begin{corollary} \label{cor realize iwahoir fix vector}
\begin{enumerate}
\item $V^{I}$ is a free $\mathcal A$-module generated by $\mathrm{ch}_I^{\psi}$. 
\item  $V^{I}$ is isomorphic to $\mathcal H \otimes_{\mathcal H_{W}} \mathrm{sgn}$.
\end{enumerate}
\end{corollary}

\begin{proof}
(1) follows from  Lemma \ref{lem computations on ch} (4) and the discussion preceding the lemma. (2) By Lemma \ref{lem computations on ch} (3)  we have 
an element in $\Hom_{\mathcal H_{W}}(\sgn, \pi^I)$ given by $1\mapsto \mathrm{ch}_I^{\psi}$ which, by Frobenius reciprocity,
furnish a map from  $\mathcal H \otimes_{\mathcal H_{W}} \mathrm{sgn}$ to $V^I$ . 
 Now (2) follows from (1) since $\mathcal H \otimes_{\mathcal H_{W}} \mathrm{sgn}$ is a free $\mathcal A$-module generated by $1\otimes 1$. 
\end{proof}

Let 
\[ 
\mathbf S= \sum_{w\in W} (-1/q)^{l(w)} T_w \in \mathcal H_W. 
\] 
If $\pi$ is a smooth representation of $G$ then $S$, acting on $\pi$,  projects on the subspace of $\pi^I$ consisting of elements on which $T_w$ act by $(-1)^{l(w)}$ for all $w\in W$. 
Let $\mathbf S(\pi)$ denote that subspace. 

Let $\widetilde \pi$ be the smooth dual of $\pi$. If $\pi$ is generated by $\pi^I$, its Iwahori-fixed vectors, then so is $\widetilde \pi$. 
 We have canonical isomorphisms $\widetilde\pi^I\cong (\pi^*)^{I}\cong (\pi^I)^*$  where $*$ denotes the linear dual.  
 In particular,  $\mathbf S(\widetilde \pi)\cong \mathbf S(\pi)^*$.
The following  is a strengthening, to the category of smooth representations,  of a genericity criteria due to Barbasch-Moy \cite{BM} for representations 
generated by Iwahori-fixed vectors. 
 
\begin{corollary} \label{cor recover results}
 Let $\pi$ be a smooth representation of $G$ generated by $I$-fixed vectors. The canonical map $\mathbf S(\pi) \rightarrow \pi_{U,\psi}$ obtained by composing 
the inclusion of $\mathbf S(\pi)$ into $\pi$ and the projection of $\pi$ onto $\pi_{U,\psi}$ is a bijection. 
\end{corollary} 
\begin{proof} It suffices to prove that the dual map $(\pi_{U,\psi})^* \rightarrow \mathbf S(\pi)^*$ is a bijection. 
We have the following natural isomorphisms:
\begin{align*}
(\pi_{U,\psi})^*&\cong  \mathrm{Hom}_{G}(\pi, \mathrm{Ind}_{U}^{G} \psi) \\
 & \cong \mathrm{Hom}_{G}(\mathrm{ind}_{U}^{G} \widetilde\psi, \widetilde{\pi}) \quad \mbox{ (taking dual) } \\
                &  \cong \mathrm{Hom}_{\mathcal H}(\mathcal H \otimes_{\mathcal H_{W}} \mathrm{sgn}, \widetilde{\pi}^{I}) \quad \mbox{ (by Corollary \ref{cor realize iwahoir fix vector})}  \\
								&  \cong \mathrm{Hom}_{\mathcal H_{W}} (\mathrm{sgn}, \widetilde{\pi}^{I})  \quad \mbox{ (by Frobenius reciprocity) } \\
								& \cong \mathbf S(\pi)^*. 
\end{align*}
It remains to show that this sequence of isomorphisms realizes the dual map $(\pi_{U,\psi})^* \rightarrow \mathbf S(\pi)^*$. To that end, let 
 $\ell \in (\pi_{U,\psi})^*$. For every $v\in \pi$, let $f_v(g) = \ell(\pi(g) v) \in \Ind_U^G\psi$. Note that $f_v(1)=\ell(v)$. 
So $\ell$ defines $A\in \Hom_G(\pi, \Ind_U^G\psi)$ by $A(v)=f_v$,  for all $v\in \pi$, and this realizes the first  isomorphism above. 
The map $A$  defines $\widetilde A \in \Hom_G(\ind_U^G\widetilde \psi,\widetilde\pi)$ where, 
for every $f\in \ind_U^G\widetilde \psi$,  $\widetilde A(f)$ is an element in $\widetilde \pi$ given by 
\[ 
\widetilde A(f)(v) = \int_{U\backslash G} f\cdot f_v ~dg
\] 
for all $v\in \pi$. This realizes the second isomorphism. The third isomorphism is given by the identification of $\ind_U^G(\widetilde\psi)^I$ and 
 $\mathcal H_n\otimes_{\mathcal H_W} \mathbb C$ where $\mathrm{ch}_{I}^{\widetilde\psi}$  corresponds to $1\otimes 1$. The fourth isomorphism gives 
an element in $ \Hom_{\mathcal H_W} (\sgn, \widetilde\pi^I)$ defined by 
 $1\mapsto \widetilde A(\mathrm{ch}_{I}^{\widetilde{\psi}})$.  
  Thus, starting from $\ell\in (\pi_{U,\psi})^*$ we have arrived to $\widetilde A(\mathrm{ch}_{I}^{\widetilde{\psi}})\in \mathbf S(\pi)^*$ given by 
\[ 
\widetilde A(\mathrm{ch}_{I}^{\widetilde{\psi}})(v) =\int_{U\backslash G} \mathrm{ch}_{I}^{\widetilde{\psi}}\cdot f_v ~dg, 
\] 
for all $v\in \mathbf S(\pi)$. Since the measure on $U\backslash G$ is fixed so that $U\cap I\backslash I$ has volume 1, the integral is equal to 
$f_v(1)$ and this is equal to $\ell(v)$, as desired. 

\end{proof}

\section{Bernstein-Zelevinsky derivatives for affine Hecke algebras} \label{s BZ afa}

In this section, we specify to $GL(n,F)$. Set $G_n=GL(n,F)$. Let $U_n$ be the unipotent subgroup of $G_n$ consisting of upper triangular matrices and let $\bar U_n$ be the opposite unipotent subgroup of $G_n$ consisting of lower triangular matrices. Let $D_n$ be the subgroup of diagonal matrices. 
The group of co-character and character lattices can be naturally identified with $X=\mathbb Z^n$.   The choice of $U_n$ determines the set of positive roots. Under 
these identifications the half-sum of all roots is $\rho=((n-1)/2, \ldots, (1-n)/2)$. 
Let $S_n$ be the group of all permutations matrices in $G_n$. 
 Let $I_n$ be the Iwahori subgroup determined from the Borel subgroup $D_n\bar U_n$ and 
 let $\mathcal H_n=C_c(I_n \backslash G_n / I_n)$ (see notations in Section \ref{ss IHA}). Inside 
$\mathcal H_n$ we have a finite dimensional subalgebra $\mathcal H_{S_n}$ consisting of functions supported on $GL(n,\mathcal O)$. Let $T_w$ be the characteristic 
function of $I_n w I_n$. Then $\mathcal H_{S_n}$ is spanned by $T_w$ for $w\in S_n$. 
Let $x=(m_1, \ldots , m_n)\in X$ such that $m_1 \geq \ldots \geq m_n$ i.e. $x$ is dominant.  Let $\dot{x}$ be the diagonal matrices whose diagonal entries are 
$\varpi^{m_1},  \ldots, \varpi^{m_n}$.
Let 
\[ 
\theta_x= q^{-\langle x, \rho\rangle}  \mathrm{ch}_{I_n \dot{x} I_n}. 
\] 
Let $\mathcal A_n$ be the commutative subalgebra in $\mathcal H_n$ generated by $\theta_x$ and their inverses, for $x$ dominant. 
 It is isomorphic to the group algebra $\mathbb C[X]$.  The algebra $\mathcal H_n$ is generated by $\mathcal H_{S_n}$ and $\mathcal A_n$ 
 modulo Bernstein's relations. 

\subsection{Jacquet functor}

We fix $i$ for the rest of this section. 
Let $P=MN$ be a parabolic subgroup containing $D_nU_n$ where $N$ is the unipotent subgroup, and the Levi subgroup 
$M\cong G_{n-i} \times G_i$ sitting in $G_n$ via the embedding
\[  (g_{n-i}, g_i)  \mapsto \begin{pmatrix} g_{n-i} & 0 \\ 0 & g_{i} \end{pmatrix}.
\]
Let $I_M=I_n \cap M$. 
 Let $\mathcal H_M= C_c(I_M\backslash M /I_M)$ be the convolution algebra of compactly supported $I_M$-bi-invariant functions on $M$. 
For every $w\in S_{n-i}\times S_i$   let $T_w^M\in \mathcal H_M$ be the characteristic function of $I_M w I_M$ 
Let $\rho_M$ be the half-sum of positive roots in $M$. Let $x\in X$ be dominant, and set 
\[ 
\theta_x^M= q^{-\langle x, \rho\rangle_M}  \mathrm{ch}_{I_M \dot{x} I_M}. 
\] 
Let $\mathcal A_M$ be a commutative subalgebra in $\mathcal H_M$ generated by $\theta_x$ and their inverses, for $x$ dominant. The following is a consequence of 
Bernstein's relations for $\mathcal H_M$ and $\mathcal H_n$. 

\begin{theorem} The map  $i_M (T_w^M)= T_w$, for $w\in S_{n-i}\times S_i$, and $i_M(\theta_x^M)=\theta_x$, for $x\in X$, defines an injective homomorphism of 
$\mathcal H_M$ and $\mathcal H_n$ 
\end{theorem} 

In particular, any $\mathcal H_n$-module $\sigma$ can be viewed as an $\mathcal H_M$-module by precomposing by $i_M$. 
 The resulting $\mathcal H_M$-module will be denoted by 
$\mathrm{res}^{\mathcal H_n}_{\mathcal H_{M}}(\sigma)$.

\begin{proposition} \label{prop Jacquet functor}
Let $\pi$ be a smooth representation of $G$. The canonical isomorphism of linear spaces $p_N : \pi^{I_n} \rightarrow (\pi_N)^{I_M}$ 
gives a canonical isomorphism of $\mathcal H_M$-modules 
\[  
\mathrm{res}^{\mathcal H_n}_{\mathcal H_{M}}(\pi^{I_n})\cong  (\pi_N)^{I_M}. 
\]
\end{proposition}
\begin{proof} 
This is proved by checking, by an explicit computation, that  $p_N \circ T_w= T_w^M \circ p_N$, for  $w\in S_{n-i}\times S_i$,  
and $p_N \circ \theta_x = \theta_x^M \circ p_N$, for dominant $x\in X$. 
\end{proof}

\subsection{Bernstein-Zelevinsky derivatives} \label{ss ZD}
We continue with the same setup. 
Let $U_{i}$ be the subgroup of $M$ consisting of matrices of the form
\[ \begin{pmatrix}  I_{n-i} & 0 \\ 0 & u \end{pmatrix} ,\]
where $u$ is a strictly upper-triangular matrix in $G_i$.  The character $\overline\psi$ of conductor $\mathfrak p$ defines a Whittaker character $\psi$ of $U_i$ 
\[ 
\psi (u) =\sum_{j=n-i+1}^{n-1}  \overline\psi(u_{j,j+1}) 
\] 
where $u_{j,j+1}$ refers to the matrix entries. Let $\sigma$ be a smooth $M$-module. Let $\sigma_{U_i, \psi}$ be
 the space of  $\psi$-twisted $U_i$-coinvaraints. It is naturally a $G_{n-i}$-module. 
If $\pi$ is a smooth $G$-module, the $i$-th Bernstein-Zelevinski derivative of $\pi$ is defined by 
\begin{align} \label{eqn BZ derivative}
 \pi^{(i)} = (\pi_N) _{U_i, \psi}
\end{align}
Thus the $i$-th Bernstein-Zelevinski derivative  is a functor from the category of smooth $G_n$-modules to the category of smooth $G_{n-i}$-modules.

\subsection{Bernstein-Zelevinsky derivative for $\mathcal H_n$} \label{ss bz der hn}
Note that we have a canonical isomorphism $\mathcal H_{n-i}\otimes \mathcal H_i \cong \mathcal H_M$ of the spaces of functions on $G_{n-i} \times G_i \cong M$. 
Composing with the injection $i_M : \mathcal H_M \rightarrow \mathcal H_n$, we have a homomorphism 
\[ 
m:  \mathcal H_{n-i}\otimes  \mathcal H_{i} \rightarrow  \mathcal H_{n}.
\] 
More concretely, we have the following formulae that will be of practical purpose later: 
$m(T_w\otimes 1)  \mapsto T_{\bar w}$, for $w\in S_{n-i}$, where $\bar w=w\times 1 \in S_{n-i} \times S_i$,  $m(\theta_x \otimes 1) \mapsto \theta_x$, 
where $x\in \mathbb Z^{n-i}$ is a viewed as an element of $\mathbb Z^n$ by 
adding $0$'s at the end, and $m(1\otimes T_w) \mapsto T_{\bar w}$, for 
$w\in S_{i}$, where $\bar w=1\times w \in S_{n-i} \times S_i$, and 
$m(1\otimes \theta_x ) \mapsto \theta_x$, where $x\in \mathbb Z^{i}$ is a viewed as an element of $\mathbb Z^n$ by 
adding $0$'s in front. 

  Abusing notation, we shall identify $\mathcal H_{n-i}$ and $m(\mathcal H_{n-i}\otimes 1)$.  Let $\mathbf S_i\in {\mathcal H}_{i}$ be the sign projector. 
 Let $\mathbf S_i^n=m(1\otimes  {\mathbf S_i})$.  
  Let $\sigma$ be an  $\mathcal H_{n}$-module.  The $i$-th Bernstein-Zelevinski derivative of $\sigma$ is the natural $\mathcal H_{n-i}$-module 
\[ 
\mathbf{BZ}_i(\sigma): =\mathbf S_i^n(\sigma). 
\]

Let $\pi$ be a smooth $G_n$-module, generated by $I_n$-fixed vectors. Then the smooth $M$-module $\pi_N$ is generated by $I_M$-fixed vectors. It is easy to see that 
$\pi_N$, viewed purely as a $G_{n-i}$-module,  is generated by its $I_{n-i}$-fixed vectors. Thus the $i$-th Bernstein-Zelevinski derivative $\pi^{(i)}$, being a quotient of 
$\pi_N$, is also generated by its $I_{n-i}$-fixed vectors. It follows that $\pi^{(i)}$ is determined by the corresponding $\mathcal H_{n-i}$-module $(\pi^{(i)})^{I_{n-i}}$. 
Now note that $(\pi^{(i)})^{I_{n-i}}$ is a quotient of $\pi^{I_{n-i}}$, while $\mathbf{BZ}_i(\pi^{I_n})$ is a submodule of $\pi^{I_{n-i}}$. Hence we have a canonical map 
$\mathbf{BZ}_i(\pi^{I_n}) \rightarrow (\pi^{(i)})^{I_{n-i}}$.

\begin{theorem} \label{thm bz aha}
Let $\pi$ be a smooth representation of $G_n$ generated by $I_n$-fixed vectors. 
The canonical  map $\mathbf{BZ}_i(\pi^{I_n}) \rightarrow (\pi^{(i)})^{I_{n-i}}$ 
is an isomorphism of $\mathcal H_{n-i}$-modules. 

\end{theorem}
\begin{proof} 
The proof of this theorem will occupy the rest of this section.

\begin{lemma} \label{lem nat is s} Let $\sigma$ be a smooth $M$-module generated by its $I_M$-fixed vectors. Then 
the canonical  map $\mathbf{S}_i(\sigma)^{I_{n-i}} \rightarrow (\sigma_{U_i, \psi})^{I_{n-i}}$ 
is an isomorphism of $\mathcal H_{n-i}$-modules. 
\end{lemma} 
\begin{proof} The canonical map is a homomorphism of $\mathcal H_{n-i}$-modules, so it suffices to check that it is an isomorphism of vector spaces. 
Note that $\sigma^{I_{n-i}}$ is generated by its $I_i$-fixed vectors as a $G_i$-module. Hence  Corollary \ref{cor recover results},  applied to $G_i$, implies the lemma. 
\end{proof} 

We now need the following observation. Let $\sigma$ be a smooth $M$-module. Then $\mathcal H_{n-i}$ and $\mathcal H_i$ both act on $\sigma^{I_M}$. The 
resulting tensor product action of $\mathcal H_{n-i}\otimes \mathcal H_i$ on $\sigma^{I_M}$ and the action of $\mathcal H_M$ are compatible with respect to 
the canonical isomorphism $\mathcal H_{n-i}\otimes \mathcal H_i \cong \mathcal H_M$.   
Using this observation  and Proposition \ref{prop Jacquet functor} one easily checks the following lemma: 

\begin{lemma} \label{lem zel der} Let $\pi$ be a smooth $G_n$-module generated by its $I_n$-fixed vectors. The isomorphism $\pi^{I_n}\cong (\pi_N)^{I_M}$ induces 
an isomorphism $\mathbf{S}_i^n(\pi^{I_n})\cong \mathbf{S}_i(\pi_N)^{I_{n-i}}$ 
of $\mathcal H_{n-i}$-modules.
\end{lemma}
\noindent 
The theorem is a simple combination of the two lemmas, using $\sigma=\pi_N$, in the first.  
\end{proof}

\section{Bernstein-Zelevinsky derivatives and Lusztig reductions} \label{s bz gha}

\subsection{Affine Hecke algebras} \label{def aha}

We shall state the definition of an affine Hecke algebra in a greater generality which will be needed in the following subsections.

Let $(X,R, X^{\vee}, R^{\vee})$ be a root datum where $R$ is a reduced root system and $X$  a $\mathbb{Z}$-lattice containing $R$. 
Let $W$ be the Weyl group of $R$. Let $Q\subseteq X$ be the root lattice and let $W_{\mathrm{aff}}=Q \rtimes W$ be the affine Weyl algebra. Fix a set of simple roots $\Pi$. The choice of $\Pi$ determines a set $S_{\mathrm{aff}}$ of simple affine reflections. 
 Let $W_{\mathrm{ex}}$ be the semidirect product $X \rtimes W$ (extended affine Weyl group).  Let $Y\subseteq X$ be the sub lattice perpendicular to $R^{\vee}$. Then 
 $W_{\mathrm{ex}}/Y$ acts on a Coxeter complex and this action defines a length function  
 $l: W_{\mathrm{aff}} \rightarrow \mathbb{Z}$ such that $l(s)=1$ for all $s\in S_{\mathrm{aff}}$. 

\begin{definition} \label{def affine heck alg}
 The affine Hecke algebra $\mathcal H:=\mathcal H(X,  R, \Pi, q)$ associated to the datum is defined to be a complex associative algebra generated by the elements $\left\{ T_w : w\in W_{\mathrm{ex}} \right\}$ subject to the relations 
\begin{enumerate}
\item $T_{w}T_{w'}=T_{ww'}$ if $l(ww')=l(w)+l(w')$,
\item $(T_s+1)(T_s-q)=0$ for $s \in S_{\mathrm{aff}}$.
\end{enumerate}
\end{definition}

Denote by $\mathcal H_W$ the finite subalgebra of $\mathcal H$ generated by $T_w$ ($w \in W$). 
The algebra $\mathcal H$ has a large commutative subalgebra $\mathcal A\cong \mathbb C[X]$, which depends on the choice of 
simple roots $\Pi$. We have an isomorphism of vector spaces $\mathcal H \cong \mathcal A\otimes_{\mathbb C} \mathcal H_W$. 
 Let $\mathbb{T}=\mathrm{Hom}(X, \mathbb{C}^{\times})$. 
The center $\mathcal Z$ of $\mathcal H$ is isomorphic  to $\mathbb C[X]^W$. Hence 
central characters of $\mathcal H$ are parameterized by $W$-orbits in $\mathbb{T}$. We shall denote by $Wt$ the $W$-orbit of $t \in \mathbb{T}$. 
Let $\mathcal J_{Wt}$ be the corresponding maximal ideal in $\mathcal Z$. 
For a finite-dimensional $\mathcal H$-module $\chi$, denote $\chi_{[Wt]}$ to be the subspace of $\chi$ annihilated by a power of 
 $\mathcal J_{Wt}$. Then 
\[  \chi \cong \bigoplus_{Wt \in \mathbb{T}/W} \chi_{[Wt]} .\]

\smallskip 
Let $X_n=X_n^{\vee} = \bigoplus_{k=1}^n \mathbb{Z}\epsilon_k$ be a $\mathbb{Z}$-lattice. Set $\alpha_{kl}=\epsilon_k-\epsilon_l$ ($k \neq l$) and also set $\alpha_k=\alpha_{k,k+1}$ ($k=1, \ldots, n$). Let $R_n=R_n^{\vee}=\left\{ \epsilon_k-\epsilon_l: l\neq k\right\}$ be a root system of type $A_{n-1}$. Let $\Pi_n=\left\{ \epsilon_i-\epsilon_{i+1}: i=1,\ldots, n-1 \right\}$. 
The Iwahori-Hecke algebra $\mathcal H_n$  of $GL(n)$ (from Section \ref{s BZ afa}) is isomorphic to $\mathcal{H}(X_n, R_n, \Pi_n, q)$.

\subsection{Lusztig's first reduction theorem} \label{s first l red}

We shall use a variation of Lusztig's reduction in \cite[Section 2]{OS}  for the affine Hecke algebra $\mathcal H_n$  (also see \cite{BM}), proofs are from \cite[Section 8]{Lu}. 
Let $\mathbb{T}_n=\mathrm{Hom}(X_n, \mathbb{C}^{\times})$. Any $t\in \mathbb{T}_n$ is identified with an $n$-tuple $(z_1, \ldots ,z_n)$ of non-zero complex numbers where 
$z_i$ is the value of $t$ at $\epsilon_i$.   Let 
 $\mathbb{T}_r=\mathrm{Hom}(X_n, \mathbb{R}_{> 0})$ and $\mathbb{T}_{un}=\mathrm{Hom}(X_n, S^1)$.   
Any $t \in \mathbb T_n$ has a polar decomposition $t=vu$ where $v \in \mathbb{T}_r$ and $u \in\mathbb{T}_{un}$. 
 Write $x(u)$ for the value of $u$ at $x\in X_n$.  Hence $u=(z_1, \ldots ,z_m)$ where $z_k=\epsilon_k(u)$. 
  Without loss of generality we can permute the entries of 
$u$ such that, for a partition $\mathbf n=(n_1, \ldots , n_m)$  of $n$, $z_1=\ldots =z_{n_1} \neq z_{n_1+1} =\ldots $ etc. Let 
\[  
R_{\mathbf n}= \left\{ \alpha \in R_n: \alpha(u)=1  \right\}.  
\] 
It is a root subsystem of $R_n$ which, as the notation indicates, depends on the partition $\mathbf n$. It is isomorphic to the product  $R_{n_1}\times \ldots \times R_{n_m}$. 
Let $S_{\mathbf n}\cong S_{n_1} \times \ldots \times S_{n_m}$ be its Weyl group. 
 Let $\Pi_{\mathbf n}$ be the set of simple roots in $R_{\mathbf n}$ determined by $R^+_{\mathbf n}=R^+_n \cap R_{\mathbf n}$. 
 Let $\mathcal H_{\mathbf n}:=\mathcal H(X_n, R_{\mathbf n}, \Pi_{\mathbf n}, q)\cong \mathcal H_{n_1} \otimes \ldots \otimes \mathcal H_{n_m}$
  be the associated affine Hecke algebra (see Definition \ref{def affine heck alg}). 
 This is a Hecke algebra corresponding to the Levi subgroup $M=G_{n_1} \times \cdots \times G_{n_m}$.  
 Let $\mathcal Z_{\mathbf n}=\mathcal A_n^{S_{\mathbf n}}$ be the center of $\mathcal H_{\mathbf n}$. Let
  $\mathcal J_{ S_{\mathbf n}t}$ be an ideal in $\mathcal Z_{\mathbf n}$ corresponding to the central character $S_{\mathbf n} t$. 
 Let 
  $\sigma$ be a finite-dimensional ${\mathcal H}_{\mathbf n}$-module annihilated by a power of $\mathcal J_{S_{\mathbf n}t}$. Then 
  $i(\sigma)=\mathcal H_n \otimes_{\mathcal H_{\mathbf n}} \sigma$ is annihilated by a power of $\mathcal J_{S_nt}$. 
 
 \begin{theorem} \label{thm equ cat first red}
 The functor $i$ defines an equivalence  between the category of finite-dimensional ${\mathcal H}_{\mathbf n}$-modules annihilated by a power of $\mathcal J_{S_{\mathbf n}t}$ and the category of finite-dimensional $\mathcal H_n$-modules annihilated by a power of $\mathcal J_{S_nt}$. 
 \end{theorem} 
 \begin{proof} 
 Let $\widehat{\mathcal Z}_n$ (depending on $S_nt$) be the $\mathcal J_{S_nt}$-adic completion of $\mathcal Z_n$. Let $\widehat{\mathcal A}_n=\widehat{\mathcal Z}_n \otimes_{\mathcal Z_n} \mathcal A_n$. Let $\widehat{\mathcal H}_n=\widehat{\mathcal Z}_n \otimes_{\mathcal Z_n} \mathcal H_n$. 
 By the Chinese Remainder Theorem for a commutative ring, we have a decomposition
\[  \widehat{\mathcal A}_n = \bigoplus_{t' \in S_nt} \widehat{\mathcal A}_{t'} ,
\]
where $\widehat{\mathcal A}_{t'}$ is obtained by localizing $\widehat{\mathcal A}_n$ at $t'$. For any $t' \in S_nt$, let $1_{t'}$ be the unit element in $\widehat{\mathcal A}_{t'}$. We also regard $1_{t'}$ as an element in $\widehat{\mathcal A}_n$. 

We define a similar formal completion of $\mathcal H_{\mathbf n}$.  
  Let $\widehat{\mathcal Z}_{\mathbf n}$ be the $\mathcal J_{S_{\mathbf n}t}$-adic completion of $\mathcal Z_{\mathbf n}$. 
  Let $\widehat{\mathcal A}_{\mathbf n}=\widehat{\mathcal Z}_{\mathbf n} \otimes_{\mathcal Z_{\mathbf n}} \mathcal A_n$. Let $\widehat{\mathcal H}_{\mathbf n} = \widehat{\mathcal Z}_{\mathbf n} \otimes_{\mathcal Z_{\mathbf n}} \mathcal H_{\mathbf n}$. 
We have a decomposition
    \[  \widehat{\mathcal A}_{\mathbf n} = \bigoplus_{t' \in S_{\mathbf n} t} \widehat{\mathcal A}_{t'} .
\]
Let $1_{\mathbf n}=\sum_{t'\in S_{\mathbf n} t} 1_{t'}$. Note that  $1_{\mathbf n}$ is in $\widehat{\mathcal Z}_{\mathbf n}$  and  
$\widehat{\mathcal A}_{\mathbf n} = 1_{\mathbf n}\cdot \widehat{\mathcal A}_n=\widehat{\mathcal A}_n\cdot 1_{\mathbf n}$.

Let $\pi$ be an $\mathcal H_n$-module annihilated by a power of $\mathcal J_{S_nt}$. Then $\pi$ is naturally an $\widehat{\mathcal H}_n$-module, and
$\sigma =1_{\mathbf n} \cdot \pi$ an ${}_{\mathbf n}\widehat{\mathcal H}_{\mathbf n}$-module, 
where ${}_{\mathbf n}\widehat{\mathcal H}_{\mathbf n}=1_{\mathbf n} \cdot \widehat{\mathcal H}_n \cdot 1_{\mathbf n}$. 
Following Lusztig's arguments \cite[Section 8]{Lu}, $  {}_{\mathbf n} \widehat{\mathcal H}_{\mathbf n} \cong \widehat{\mathcal H}_{\mathbf n}$. 
Hence by identifying $ {}_{\mathbf n} \widehat{\mathcal H}_{\mathbf n} \cong \widehat{\mathcal H}_{\mathbf n}$, we have a functor $r(\pi)=1_{\mathbf n}\cdot \pi$ from the category of finite-dimensional $\mathcal H_n$-module annihilated by a power of $\mathcal J_{S_nt}$ to the category of finite-dimensional $\mathcal H_{\mathbf n}$-modules annihilated by a power of $\mathcal J_{S_{\mathbf n}t}$. Using the Frobenius reciprocity, intertwining operators (see \cite[Lemma 8.9(a)]{Lu}) and the fact that $1=\sum_{t'\in S_nt}1_{t'}$, we obtain a natural isomorphism from $i\circ r(\pi)$ to $\pi$. Using intertwining operators (see \cite[Lemma 8.9(a)]{Lu}) and the fact that
 $1_{\mathbf n}\cdot 1_{t'}=0$ if $t'\notin S_{\mathbf n}t$, we obtain $r \circ i \cong \mathrm{Id}$.  Hence $i$ defines an equivalence of categories. 
\end{proof}

\subsection{First reduction for the Bernstein-Zelevinsky derivatives} \label{s translate BZ}

We keep using notations from the previous subsection. In particular, we fixed $t=vu \in \mathbb{T}_n$, and 
 we have a canonical isomorphism $\mathcal H_{\mathbf n} \cong \mathcal H_{n_1} \otimes \ldots \otimes \mathcal H_{n_m}$, 
 where $\mathbf n=(n_1, \ldots, n_m)$ is a partition of $n$, arising from $u$. 
 

Fix an integer $i \leq n$. For each $m$-tuple  $\mathbf i=(i_1,\ldots, i_m)$  of integers, such that $i_1+\ldots +i_m=i$ and $0\leq i_k\leq n_k$ ($k=1,\ldots, m$), define
another $m$-tuple $\mathbf n-\mathbf i=(n_1- i_1,\ldots, n_m-i_m)$.  Each pair $(n_k-i_k, i_k)$ gives rise to an embedding 
$\mathcal H_{n_k-i_k} \otimes \mathcal H_{i_k} \subseteq \mathcal H_{n_k}$, as in Section \ref{ss bz der hn}, and these combine to give an embedding 
\[ 
\mathcal H_{\mathbf n-\mathbf i} \otimes \mathcal H_{\mathbf i} \subseteq \mathcal H_{\mathbf n}
\] 
where $\mathcal H_{\mathbf i} \cong \mathcal H_{i_1} \otimes \ldots \otimes \mathcal H_{i_m}$ etc. (Note, if $i_k=0$, then the corresponding factor is the trivial 
algebra $\mathbb C$.) Abusing notation, we shall identify $\mathcal H_{\mathbf n-\mathbf i}$ with its image in $\mathcal H_{\mathbf n}$ via the map $h\mapsto h\otimes 1$. 
Let $\mathbf S_{\mathbf i}\in \mathcal H_{\mathbf i}$ be the sign projector in $\mathcal H_{\mathbf i}$, 
and let $\mathbf S_{\mathbf i}^{\mathbf n}$ be the image of $1\otimes \mathbf S_{\mathbf i}$ in $\mathcal H_{\mathbf n}$.  Let $\sigma$ be an $\mathcal H_{\mathbf n}$-module. 
Then $\mathbf S_{\mathbf i}^{\mathbf n}(\sigma)$ is naturally an $\mathcal H_{\mathbf n-\mathbf i}$-module. Thus we have a functor 
\[  \mathbf{BZ}^{\mathbf n}_{\mathbf i}(\sigma): =\mathbf S_{\mathbf i}^{\mathbf n}(\sigma) 
\]
 from the category of $\mathcal H_{\mathbf n}$-modules to the category of ${\mathcal H}_{\mathbf n-\mathbf i}$-modules. 

Observe that $\mathcal H_{\mathbf n-\mathbf i}$ is a Levi subalgebra of $\mathcal H_{n-i}$ and 
$\mathcal H_{\mathbf i}$ is a Levi subalgebra of $\mathcal H_{i}$
We are now ready to state the first reduction result.


\begin{theorem} \label{thm first red bz}
Let $\pi$ be a finite-dimensional $\mathcal H_n$-module annihilated by a power of $\mathcal J_{S_nt}$. Let $\sigma$ be a finite-dimensional $\mathcal H_{\mathbf n}$-module annihilated by a power of $\mathcal J_{S_{\mathbf n}t}$ such that $\pi \cong i(\sigma)$ (see Theorem \ref{thm equ cat first red}). Then there is an isomorphism
\begin{align} \label{eqn bz first red} 
\mathbf{BZ}_i(\pi) \cong \bigoplus_{\mathbf i} \mathcal H_{n-i} \otimes_ {{\mathcal H}_{\mathbf n-\mathbf i}}    \mathbf{BZ}^{\mathbf n}_{\mathbf i} ( \sigma ) 
\end{align}
where the sum is taken over all $m$-tuple of integers $\mathbf i=(i_1,\ldots, i_m)$ satisfying $i_1+\ldots +i_m=i$ and $0\leq i_k\leq n_k$ ($k=1,\ldots, m$).
\end{theorem}

\begin{proof}
By using the Mackey theorem for affine Hecke algebras (see e.g. \cite[Section 3.5]{Kl} for a similar setting), we have
\begin{align}\label{eqn mackey thm} 
\mathrm{res}^{\mathcal H_n}_{\mathcal H_{n-i} \otimes \mathcal H_i} (\mathcal H_{n} \otimes_{\mathcal H_{\mathbf n}} \sigma) \cong 
\bigoplus_{\mathbf i} (\mathcal H_{n- i} \otimes \mathcal H_{ i}) 
\otimes_{(\mathcal H_{\mathbf n-\mathbf i} \otimes \mathcal H_{\mathbf i})} 
 \left(\mathrm{res}^{\mathcal H_{\mathbf n}}_{\mathcal H_{\mathbf n-\mathbf i} \otimes \mathcal H_{\mathbf i}}  \sigma \right)
\end{align}
where the sum is over $\mathbf i$ as in the statement of the theorem. 
We remark that the Mackey Theorem asserts that the composition factors of $ \mathrm{res}^{\mathcal H_n}_{\mathcal H_{n-i} \otimes \mathcal H_i} (\mathcal H_{n} \otimes_{\mathcal H_u} \sigma)$ are of the form in the left hand side of the above isomorphism. Those composition factors are indeed direct summands since the $\mathcal H_{n-i} \otimes \mathcal H_i$-central characters of those composition factors are distinct. 
Furthermore, using the Frobenius reciprocity, we have
\begin{align}\label{eqn sign first red} \mathbf S_i^n ((\mathcal H_{n-i} \otimes \mathcal H_{i}) \otimes_{(\mathcal H_{\mathbf n-\mathbf i} \otimes \mathcal H_{\mathbf i})} \sigma )\cong  \mathcal H_{n-i} \otimes_{{\mathcal H}_{\mathbf n-\mathbf i} } \mathbf S_{\mathbf i}^{\mathbf n}(\sigma). 
\end{align}
 Combining (\ref{eqn mackey thm}) and (\ref{eqn sign first red}), we obtain (\ref{eqn bz first red}).
\end{proof}

\begin{remark}
When $\pi \cong i(\sigma)$ is an irreducible $\mathcal H_{\mathbf n}$-module, then $\sigma \cong \sigma_1 \boxtimes \ldots \boxtimes \sigma_m$ for some irreducible $\mathcal H_{i_k}$-modules $\sigma_k$. In this case, 
\[\mathbf{BZ}^{\mathbf n}_{\mathbf i} ( \sigma ) \cong \mathbf{BZ}_{i_1}(\sigma_1) \boxtimes \ldots \boxtimes \mathbf{BZ}_{i_m}(\sigma_m) .
\]
From this viewpoint, Theorem \ref{thm first red bz} can be seen as a Leibniz rule. 
\end{remark}

\subsection{Graded affine Hecke algebras } 

We shall now need the affine graded Hecke algebra attached to the root datum $(X,R,X^{\vee},R^{\vee})$. 
   Let $V =X\otimes_{\mathbf Z}  \mathbb{C}$.

\begin{definition} \label{def gah}
\cite[Section 4]{Lu}
The graded affine Hecke algebra $\mathbb{H}=\mathbb{H}(V, R, \Pi, \log q)$ is an associative algebra with an unit over $\mathbb{C}$ generated by the symbols $\left\{ t_w :w \in W \right\}$ and $\left\{ f_v: v \in V \right\}$ satisfying the following relations:
\begin{enumerate}
\item[(1)] The map $w  \mapsto t_w$ from $\mathbb{C}[W]=\bigoplus_{w\in W} \mathbb{C}w  \rightarrow \mathbb{H }$ is an algebra injection,
\item[(2)] The map $v \mapsto f_v$ from $S(V) \rightarrow \mathbb{H}$ is an algebra injection, where $S(V)$ is the polynomial ring for $V$, 
\item[(3)] writing $v$ for $f_v$ from now on, for $\alpha \in \Pi$ and $v \in V$,
\[    vt_{s_{\alpha}}-t_{s_{\alpha}}s_{\alpha}(v)=\log q \cdot \langle v, \alpha^{\vee} \rangle .\]
\end{enumerate}
\end{definition}

In particular, $\mathbb{H} \cong S(V) \otimes \mathbb{C}[W]$ as vector spaces.
We also set $\mathbb{A}=S(V)$, the graded algebra analogue of $\mathcal A$. Let $\mathbb{Z}=\mathbb{A}^{W}$ be the center of $\mathbb{H}$. 
Let $V^*=\Hom(X, \mathbb C)$. The central characters of irreducible representations are parameterized by $W$-orbits in $V^*$. If 
$\zeta\in V^*$, let $W\zeta$ denote the corresponding orbit an the central character. Let $\mathbb J_{W\zeta} \subset \mathbb Z$ be the corresponding 
maximal ideal.

\subsection{Lusztig's second reduction theorem} \label{ss lus srt}

 Let $\mathcal H=\mathcal H( X, R, \Pi,  q)$ be the affine Hecke algebra defined in Section \ref{def aha}, and 
 $\mathcal A\cong \mathbb C[X]$ the commutative sub algebra. Let $\theta_{x}\in \mathcal A$ correspond to $x\in X$.  
 Let $\mathcal Z\cong \mathbb C[X]^W$ be the center of $\mathcal H$. Let $\mathcal F$ be the quotient field of $\mathcal A$. Let $\mathcal H_F \cong \mathcal H_W \otimes \mathcal F$ with the algebraic structure naturally extending $\mathcal H$.

Following Lusztig \cite[Section 5]{Lu}, for $\alpha\in\Pi$, define $\tau_{s_{\alpha}} \in \mathcal H_F$ by 
\[  \tau_{s_{\alpha}}+1=(T_{s_{\alpha}}+1)\mathcal G(\alpha)^{-1} ,\]
where 
\[  \mathcal G(\alpha) = \frac{\theta_{\alpha}q-1}{\theta_{\alpha}-1} \in \mathcal F.
\]
It is shown in \cite[Section 5]{Lu} that the map from $W$ to the units of $\mathcal H_F$ defined by $s_{\alpha} \mapsto \tau_{s_{\alpha}}$ is an injective group homomorphism. 

On the graded Hecke algebra side, let $\mathbb{H}=\mathbb{H}(V, R, \Pi, \log q)$ be as in Definition \ref{def gah}. Let $\mathbb{F}$ be the quotient field of $\mathbb A$ and let $\mathbb{Z}$ be the center of $\mathbb{H}$. Let $\mathbb H_F \cong \mathbb H_W \otimes \mathbb F$ with the algebraic structure naturally extending $\mathbb H$. For $\alpha \in \Pi$, define $\overline{\tau}_{s_{\alpha}} \in \mathbb H_F$ by 
\[  \overline{\tau}_{s_{\alpha}}+1=(t_{s_{\alpha}}+1)g(\alpha)^{-1}, \]
where
\[  g(\alpha)= \frac{\alpha+\log q}{\alpha} \in \mathbb F . 
\]
As in the affine case, the map from $W$ to the units of $\mathbb H_F$ defined by $s_{\alpha} \mapsto \overline{\tau}_{s_{\alpha}}$ is an injective group homomorphism.

 Any  $\zeta \in V^*$ defines $t \in \mathbb{T}=\mathrm{Hom}(X, \mathbb{C}^{\times})$ by 
 $x(t)=e^{x(\zeta)}$, for all $x\in X$.  We shall express this relationship by $t=\exp(\zeta)$. We shall say that  $\zeta$ is {\em real} for the root system $R$ if 
   $\alpha(\zeta) \in \mathbb R$ for all $\alpha\in R$. Then $t=\exp(\zeta)$ satisfies 
  $\alpha(t) >0$, for all $\alpha\in R$. Conversely, every such $t$ arises in this fashion, from a real $\zeta$. 
Let $\widehat{\mathcal Z}$ be the $\mathcal J_{Wt}$-adic completion of $\mathcal Z$ and let $\widehat{\mathbb{Z}}$ be the $\mathbb{J}_{W\zeta}$-adic completion of $\mathbb{Z}$. Let $\widehat{\mathcal H}=\widehat{\mathcal Z} \otimes_{\mathcal Z} \mathcal H$ and let  $\widehat{\mathbb H}=\widehat{\mathbb Z} \otimes_{\mathbb Z} \mathbb H$. Let $\widehat{\mathcal H}_F=\widehat{\mathcal Z} \otimes_{\mathcal Z} \mathcal H_F$ and let  $\widehat{\mathbb H}_F=\widehat{\mathbb Z} \otimes_{\mathbb Z} \mathbb H_F$. Let $\widehat{\mathcal A}=\widehat{\mathcal Z} \otimes_{\mathcal Z} \mathcal A$ and let  $\widehat{\mathbb A}=\widehat{\mathbb Z} \otimes_{\mathbb Z} \mathbb A$. Let $\widehat{\mathcal J}_{Wt}=\widehat{\mathcal Z}\otimes_{\mathcal Z} \mathcal J_{Wt}$ and let $\widehat{\mathbb J}_{W\zeta}=\widehat{\mathbb Z}\otimes_{\mathbb Z}\mathbb J_{W\zeta}$.



\begin{theorem} \label{thm lusztig red thm}\cite[Theorem 9.3, Section 9.6]{Lu} Recall that we are assuming  that $\zeta\in V^*$ is real for the root system $R$. 
\begin{enumerate}
\item There is an isomorphism denoted $j$ between $\widehat{\mathcal H}_F$ and $\widehat{\mathbb{H}}_F$ determined by 
\[  j(\tau_{s_{\alpha}}) = \overline{\tau}_{s_{\alpha}} ,   \quad   j(\theta_{x})=e^{x}. \] 
\item The above map also induces isomorphisms between $\widehat{\mathcal Z}$ and $\widehat{\mathbb Z}$, between $ \widehat{\mathcal A}$ and $\widehat{\mathbb A}$ and between $\widehat{\mathcal H}$ and $\widehat{\mathbb{H}}$. 
\end{enumerate}
\end{theorem}

A crucial point for the proof of (2) is the fact that
\[ \frac{e^{\alpha}q-1}{e^{\alpha}-1} \cdot\frac{\alpha}{\alpha+\log q} \in \mathbb{F}
\]
is holomorphic and nonvanishing at any $\zeta' \in W\zeta$, and hence is an invertible element in $\widehat{\mathbb{A}}$. 

Now (2) gives the following isomorphisms:
\[ \mathcal H/\mathcal J_{Wt}^i\mathcal H \cong \widehat{\mathcal H}/ \widehat{\mathcal J}_{Wt}^i\widehat{\mathcal H} \cong \widehat{\mathbb H}/ \widehat{\mathbb J}_{W\zeta}^i\widehat{\mathbb H} \cong \mathbb H/\mathbb J_{W\zeta}^i\mathbb H 
\]
and hence:

\begin{theorem} \label{thm equiv cat} \cite[Section 10]{Lu} Assume that $\zeta\in V^*$ is real. 
There is an equivalence of categories between the category of finite-dimensional 
$\mathbb{H}$-modules annihilated by a power of $\mathbb J_{W\zeta}$
and the category of finite-dimensional $\mathcal H$-modules annihilated by a power of $\mathcal J_{Wt}$, where $t=\exp(\zeta)$. 
\end{theorem}

Let $\Lambda$ be the functor in Theorem \ref{thm equiv cat}.  Explicitly, for a finite-dimensional $\mathbb H$-module annihilated by a power of $\mathbb J_{W\zeta}$, 
$\Lambda(\pi)$ is equal to $\pi$, as linear spaces,  but the $\mathcal H$-action on $\pi$ is given by 
\[   h \cdot_{\mathcal {H}} x = j (h) \cdot_{\widehat{\mathbb H}} x,
\]
where $h \in \mathcal{H}$ and $x\in \pi$. Note that the functor extends to the category of finite dimensional $\mathbb H$-modules that are sums of $\mathbb H$-modules, 
where each summand is annihilated by a power of $\mathbb J_{W\zeta}$ for some real $\zeta$. 


\begin{proposition} \label{prop translate sign} Recall the sign projector $\mathbf S=\sum_{w \in W} (-1/q)^{l(w)} T_w$ in $\mathcal H$ and let 
$\mathbf s=\sum_{w \in W} (-1)^{l(w)} t_w$ be the corresponding sign projector in $\mathbb H$.  Then $j(\mathbf S)=a\cdot \mathbf s$, 
where $a$ is an invertible element in $ \widehat{\mathbb A}$.  

\end{proposition}

\begin{proof} Let $\alpha\in \Pi$. 
Firstly, by a direct computation, we have
\[ 
j( 1-q^{-1} T_{s_{\alpha}})= j(\mathcal G(-\alpha))^{-1} g(-\alpha) q^{-1} (1-t_{s_{\alpha}}).  
\] 
Secondly, 
\[  
 \mathbf S=\left(\sum_{w \in W^{\Pi \setminus \left\{ \alpha \right\}}} (-1/q)^{l(w)} T_{w}\right) (1-q^{-1}T_{s_{\alpha}}) ,
\]
where $W^{\Pi\setminus \left\{ \alpha \right\}}$ is the set of minimal representatives of $W/W_{\Pi \setminus \left\{ \alpha \right\}}$ and $W_{\Pi \setminus \left\{ \alpha \right\}}$ is the parabolic subgroup associated to $\Pi \setminus \left\{ \alpha \right\}$. Therefore 
\[
j(\mathbf S)=  j\left(\sum_{w \in W^{\Pi \setminus \left\{ \alpha \right\}}} (-1/q)^{l(w)} T_w \right)j({\mathcal G}(-\alpha))^{-1}g(-\alpha)q^{-1}(1-t_{s_{\alpha}}).\] 
 Hence we have $j(\mathbf S)t_{s_{\alpha}}=-j(\mathbf S)$. This shows that $j(\mathbf S)\in \widehat{\mathbb H}\cdot  {\mathbf s}$. 
 Since $\widehat{\mathbb H}\cdot  {\mathbf s}=\widehat{\mathbb A}\cdot  {\mathbf s}$, we have  $j(\mathbf S)=a\cdot \mathbf s$, for some $a\in \widehat{\mathbb A}$. 
  Using the same argument, for $j^{-1}$, we obtain $j^{-1}(\mathbf s)= b \cdot \mathbf S$ for some $b\in \widehat{\mathcal A}$.  Hence $j(b)a=1$  and $a$ is invertible. 
\end{proof}

We have the following corollary to Proposition \ref{prop translate sign}: 

\begin{corollary} \label{cor equality sign }
Let $\pi$ be a finite dimensional  $\mathbb H$-module annihilated by a power of $\mathcal J_{W\zeta}$, where $\zeta\in V^*$ is real. 
 Identify $\pi$ and $\Lambda(\pi)$ as linear spaces. 
The multiplication by $a\in \widehat{\mathbb A}$ (from Proposition \ref{prop translate sign})  provides a natural isomorphism between the linear spaces  $ \mathbf{s}(\pi)$ and $\mathbf S(\Lambda(\pi))$. 
\end{corollary}

\subsection{Bernstein-Zelevinsky derivatives for graded algebras} 
Let $V_n=X_n \otimes_{\mathbb Z}  \mathbb C$, and $\mathbb{H}_n:=\mathbb{H}(V_n, R_n, \Pi_n, \log q)$. For every $i=0, \ldots, n$, we have a Levi subalgebra 
$\mathbb H_{n-i} \otimes \mathbb H_i$.  Let $\mathbf s_i\in \mathbb H_i$ be the sign projector, and let $\mathbf s^n_i \in \mathbb H_n$ be 
the image of $1\otimes \mathbf s_i$ under the inclusion $\mathbb H_{n-i} \otimes \mathbb H_i\subseteq \mathbb H_n$. 

Let $\pi$ be a finite dimensional representation of $\mathbb H_n$.  The $i$-the Bernstein-Zelevinsky derivative of $\pi$ is the natural 
$\mathbb H_{n-i}$-module 
\[ 
\mathbf{gBZ}_i(\pi):=\mathbf s_i^n(\pi). 
\] 

Write any $\zeta\in V_n^*= \Hom(X_n, \mathbb C)$ as an $n$-tuple $(\zeta_1, \ldots, \zeta_n)$ where $\zeta_i$ is the value of $\zeta$ on 
the standard basis element $\epsilon_i\in X_n$. In this case $\zeta$ is real for $R_n$ if and only if $\zeta_k-\zeta_l \in \mathbb R$ for all $1\leq k,l\leq n$.  

\begin{theorem} \label{thm real bzder}  Assume that $\zeta\in V_n^*$ is real for the root system $R_n$, and $\pi$ is a finite-dimensional $\mathbb H_n$-module annihilated by a power 
of $\mathbb J_{S_n\zeta}$. There is a natural isomorphism of $\mathcal H_{n-i}$-modules 
$\mathbf{BZ}_i(\Lambda(\pi))$ and $\Lambda(\mathbf{gBZ}_i(\pi))$. 
\end{theorem}
\begin{proof} Note that the functor $\Lambda$ commutes with the restriction to Levi subalgebras, that is, we can either restrict to $\mathbb H_{n-i} \otimes \mathbb H_i$ and 
then apply $\Lambda$, or apply $\Lambda$ and then restrict to $\mathcal H_{n-i} \otimes \mathcal H_i$. Decompose $\pi$ under the action of $\mathbb H_i$ 
\[ 
\pi =\oplus \pi_{[S_i \zeta']} 
\] 
where $\pi_{S_i \zeta'}$ is the summand annihilated by a power of $\mathbb J_{S_i\zeta'}$. Concretely, the sum runs over $S_i$-orbits of the $i$-tuples $\zeta'$ that appear 
as the tail end of the $n$-tuples in the $S_n$-orbit of $\zeta$. We have the corresponding decomposition for the action of $\mathcal H_i$, 
\[ 
\Lambda(\pi) =\oplus \Lambda(\pi)_{[S_i t']} 
\] 
where $t'=\exp(\zeta')$. (The underlying vector spaces of $\pi_{[S_i \zeta']} $ and $\Lambda(\pi)_{[S_i t']}$ are the same.) It follows that $\Lambda(\pi)_{[S_i t']}$ and $\Lambda( \pi_{[S_i \zeta']} )$ are isomorphic $\mathcal H_{n-i} \otimes \mathcal H_i$-modules. 
Recall that $\mathbf S_i^n= 1\otimes \mathbf S_i$ and $ \mathbf s_i^n=1\otimes \mathbf s_i$, where $\mathbf S_i$ and $\mathbf s_i$ are the sign projectors in 
$\mathcal H_i$ and $\mathbb H_i$, respectively. 
Now we have the following isomorphisms of $\mathcal H_{n-i}$-modules 
\[ 
\mathbf S_i^n(\Lambda(\pi)_{[S_i t']}) \cong \mathbf S_i^n( \Lambda( \pi_{[S_i \zeta']} )\cong \Lambda( \mathbf s_i^n(\pi_{[S_i\zeta']}))
\] 
where the second is furnished by Corollary \ref{cor equality sign }.  This isomorphism is given by the action of an invertible element in 
$\widehat{\mathbb H}_i$ and therefore intertwines $\mathcal H_{n-i}$-action. 
\end{proof}

\subsection{Second reduction for Bernstein-Zelevinsky derivatives}

In this section, we transfer the problem of computing Bernstein-Zelevinsky derivatives $\mathbf{BZ}_{\mathbf i}^{\mathbf n}$ in Theorem \ref{thm first red bz} to the corresponding problem for graded Hecke algebras. We retain the notations in Sections \ref{s first l red}  and \ref{s translate BZ}. In particular, $\mathbf n=(n_1, \ldots, n_m)$ is a 
partition of $n$, and  we have fixed $t\in \mathbb{T}_n$ such that  $\alpha(t)>0$ for all $\alpha \in R_{\mathbf n}$. 
Then there exists $\zeta\in V_n^*$, real  for the root system $R_{\mathbf n}$, such that $t=\exp(\zeta)$. 
Let 
\[ \mathbb H_{\mathbf n}:= \mathbb{H}(V_n, R_{\mathbf n}, \Pi_{\mathbf n},  \log q)\cong \mathbb H_{n_1} \otimes \ldots \otimes \mathbb H_{n_m}.
\]
Let $\mathbf i=(i_1, \ldots, i_m)$ be an $m$-tuple of integers such that $0\leq i_k\leq n_k$ for all $k$ and 
 $\mathbf n-\mathbf i=(n_1- i_1,\ldots, n_m-i_m)$.  Each pair $(n_k-i_k, i_k)$ gives rise to an embedding 
$\mathbb H_{n_k-i_k} \otimes \mathbb H_{i_k} \subseteq \mathbb H_{n_k}$, and these combine to give an embedding 
\[ 
\mathbb H_{\mathbf n-\mathbf i} \otimes \mathbb H_{\mathbf i} \subseteq \mathbb H_{\mathbf n}
\] 
where $\mathbb H_{\mathbf i} \cong \mathbb H_{i_1} \otimes \ldots \otimes \mathbb H_{i_m}$ etc. 
Abusing notation, we shall identify $\mathbb H_{\mathbf n-\mathbf i}$ with its image in $\mathbb H_{\mathbf n}$ via the map $h\mapsto h\otimes 1$. 
Let $\mathbf s_{\mathbf i}\in \mathbb H_{\mathbf i}$ be the sign projector in $\mathbb H_{\mathbf i}$, 
and let $\mathbf s_{\mathbf i}^{\mathbf n}$ be the image of $1\otimes \mathbf s_{\mathbf i}$ in $\mathbb H_{\mathbf n}$.  Let $\sigma$ be an $\mathbb H_{\mathbf n}$-module. 
Then $\mathbf s_{\mathbf i}^{\mathbf n}(\sigma)$ is naturally an $\mathbb H_{\mathbf n-\mathbf i}$-module. Thus we have a functor 
\[ 
 \mathbf{gBZ}^{\mathbf n}_{\mathbf i}(\sigma): =\mathbf s_{\mathbf i}^{\mathbf n}(\sigma) 
\]
 from the category of $\mathbb H_{\mathbf n}$-modules to the category of ${\mathbb H}_{\mathbf n-\mathbf i}$-modules. The following is proved in the 
 same way as Theorem \ref{thm real bzder}.

\begin{theorem} \label{thm bz derivative graded case} Let $\zeta\in V^*_n$ be real for the root system $R_{\mathbf n}$. 
 Let $\pi$ be a finite-dimensional $\mathbb{H}_{\mathbf n}$-module annihilated by a power of $\mathbb J_{S_{\mathbf n}\zeta}$.
 Then we have a natural isomorphism of $\mathcal H_{\mathbf n-\mathbf i}$-modules 
\[   (\mathbf{BZ}_{\mathbf i}^{\mathbf n}(\Lambda(\pi)) \cong \Lambda(\mathbf{gBZ}_{\mathbf i}^{\mathbf n}((\pi)) . 
\]

\end{theorem}

\section{Bernstein-Zelevinsky derivatives of Speh representations} \label{s bz speh}

\subsection{Speh modules}

Speh representations of $p$-adic groups were studied extensively by Tadi\'c as a part of studying the unitary dual. We recall the definition of 
(generalized) Speh representations. Let $\bar{n}$ be a partition of $n$, write $\bar{n}^t= (e_1, \ldots ,e_f)$, $e_1\geq \ldots \geq e_f$,  where $t$ is the transpose. 
Let $\mathrm{St}_{e_k}$ be the Steinberg representation of $GL(e_k,F)$ and let $\mathrm{St}_{e_k}'=\nu^{-\frac{e_k-1}{2}}\mathrm{St}_{e_k}$ be a twist of $\mathrm{St}_{e_k}$, where $\nu(g)=|\mathrm{det}(g)|_F$. Let $P_{\bar{n}}$ be the standard parabolic subgroup associated to the partition $\bar{n}^t$. Let $\rho(g)=|\mathrm{det}(g)|_F^r$ for some complex number $r$. The unique quotient of the induced representation 
\[ 
\pi_{(\bar{n},\rho)}=\mathrm{Ind}_{P_{\bar{n}}}^{GL(n,F)} (\rho\mathrm{St}_{e_1}' \boxtimes \rho\nu\mathrm{St}_{e_2}' \cdots \boxtimes \rho\nu^{f-1}\mathrm{St}_{e_f}')
\] 
 is the generalized Speh representation associated to ($\bar{n}, \rho$).  If $e_1=e_2=\ldots =e_f$ then $\pi_{\bar{n}}$ is a Speh representation. 
 

Under the Borel-Casselman equivalence, generalized Speh representations correspond to $\mathcal H_n$-modules with single $\mathcal H_{S_n}$-type (see \cite{BC}, \cite{BM2}, \cite{CM}). 
Since these $\mathcal H_n$-modules have real infinitesimal character, we can look at the corresponding modules for the graded algebra $\mathbb{H}_n$. 
 They can be intrinsically constructed as follows. For $\kappa=-r \log q$, we have the following Jucys-Murphy elements: for $k=2,\ldots, n$,
\begin{align} \label{eqn jucys murphy}
 JM_k:= - p( t_{s_{1,k}}+\dots + t_{s_{k-1,k}}) +\kappa
\end{align}
and $JM_1=\kappa$, where $p=\log q$. 
It is straightforward to check that the maps $\epsilon_k \mapsto JM_k$ and $t_{w} \mapsto t_w$ define an algebra homomorphism from $\mathbb{H}_n$ to $\mathbb{C}[S_n]$. Let $\sigma_{\bar{n}}$ be the irreducible $\mathbb{C}[S_n]$-module corresponding to $\bar{n}$. For example, the partition $(n)$ defines the trivial representation while $(1,\ldots, 1)$ defines the sign representation.
 Let $\sigma_{(\bar{n},\kappa)}$ be the $\mathbb{H}$-module pulled back from $\sigma_{\bar{n}}$ via the map defined above, where $JM_k$ depends on $\kappa$. 
 This is the  {\it generalized Speh module} associated to $(\bar{n},\kappa)$. The module $\sigma_{(\bar{n},\kappa)}$ corresponds to $\pi_{(\bar{n},\rho)}$ under the 
 Borel-Casselman equivalence and the Lusztig equivalence in Theorem \ref{thm equiv cat}.

Recall that $\mathbf{gBZ}_i(\pi)$ is the $i$-the Bernstein-Zelevinsky derivative of an $\mathbb{H}_n$-module $\pi$.

\begin{lemma} \label{lem restrict speh}
Let $\pi$ be the generalized Speh $\mathbb H_n$-module associated to the datum $(\bar{n},\kappa)$. 
Then $\mathbf{gBZ}_{i}(\pi)$ is a direct sum of generalized Speh $\mathbb H_{n-i}$-modules. 
Moreover, $\epsilon_1$ acts by the constant $\kappa$ on each direct summand of $\mathbf{gBZ}_{i}(\pi)$.
\end{lemma}
\begin{proof}
This follows from the construction of generalized Speh modules (see e.g. (\ref{eqn jucys murphy})) and the fact that the category of $\mathbb{C}[S_n]$-modules is semisimple.
\end{proof}

We now recover a result of Lapid-M\'inguez (for the case of generalized Speh modules).

\begin{corollary} \label{cor bz speh}
Let $\pi$ be a generalized Speh representation of $GL(n,F)$ associated to $(\bar{n}, \rho)$. Then $\pi^{(i)}$ is the direct sum of generalized Speh modules associated to 
$(\bar{n}',\rho)$, where $\bar{n}'$ runs for all the partitions obtained by removing $i$ boxes from $\bar{n}$ with at most one in each row such that the resulting diagram is still a Young diagram.

\end{corollary}

\begin{proof}
Since $\Lambda(\sigma_{\bar{n},\kappa})=\pi^{I_n}$ 
  it suffices to compute $\mathbf{gBZ}_i(\sigma_{\bar{n},\kappa})$ by Theorem \ref{thm real bzder}. 
From the observation in Lemma \ref{lem restrict speh}, it suffices to determine the $\mathbb{C}[S_{n-i}]$-module structure of $\mathbf{gBZ}_i(\sigma_{\bar{n},\kappa})$, 
 and this follows from a special case of the Littlewood-Richardson rule (or the Pieri's formula). 
\end{proof}

Generalized Speh modules form a subclass of ladder representations defined by Lapid-M\'inguez \cite{LM}. Bernstein-Zelevinsky derivatives of ladder representations are 
computed there using a determinantal formula of Tadi\'c.

\section{branching rules and Locally nice representations} \label{s ext branch}

\subsection{Bernstein-Zelevinsky filtration}
 Let $E_n$ be the mirabolic subgroup of $GL(n+1,F)$ i.e. the subgroup of all matrices of the form $\begin{pmatrix} g & v \\ 0 &1 \end{pmatrix}$, where $g \in GL(n,F)$ and $v \in M_{n \times 1}$. For $i=1,\ldots, n+1$ let 
 \[ 
 R_i=\left\{ \begin{pmatrix} g & v \\ 0 & u \end{pmatrix} : g \in GL(n+1-i,F), v \in M_{n+1-i,i}, u \in U_{i} \right\}.
 \]  
We recall a result of Bernstein-Zelevinsky:

\begin{theorem}\label{thm bz filtration}
Let $(\pi,X)$ be a smooth representation of $GL(n+1,F)$. Then, as a representation of $E_n$, $\pi$ admits a filtration
\begin{equation} \label{eqn bz filtration}
  0 =X_{n+1} \subset X_n  \subset \ldots \subset X_1 \subset X_{0}=X 
\end{equation}
such that for $i=1,\ldots, n+1$
\[   X_{i-1}/X_{i} \cong \mathrm{ind}_{R_{i}}^{E_{n}}(\pi^{(i)} \boxtimes \psi_{i} ),
\]
$\pi^{(i)}$ is the $i$-th Bernstein-Zelevinsky derivative,  and $\psi_{i}$ is the Whittaker character for $U_{i}$. 
\end{theorem}

We abbreviate $G_n=GL(n,F)$ etc. 
Since $E_n = G_n R_i$, any element in $\mathrm{ind}_{R_i}^{E_n}(\pi^{(i)} \boxtimes \psi_{i})$ is determined by its restriction to $G_n$. 
 Hence, for $i \geq 1$, the restriction of functions  defines an isomorphism of $G_n$-modules, 
 \[ 
 \mathrm{ind}_{R_i}^{E_n}(\pi^{(i)} \boxtimes \psi_{i})\cong \mathrm{ind}_{Q_i}^{G_n} (\nu^{\frac{1}{2}}\pi^{(i)} \boxtimes \psi_{i-1}),
 \]  where $Q_i=R_i\cap G_n$ and  $\nu(g)=|\mathrm{det}(g)|_F$. 

Let $P_i=M_i N_i$ be the maximal parabolic consisting of block upper triangular matrices in $G_n$  with the Levi factor $M_i=G_{n+1-i} \times G_{i-1}$ of block diagonal 
matrices. In particular, $P_i$ contains $Q_i$.  Fix an embedding of $\mathcal H_{n+1-i}\otimes \mathcal H_{i-1}$ into $\mathcal H_n$ such that the restriction functor from 
the category of $\mathcal H_n$-modules to the category of $\mathcal H_{n+1-i}\otimes \mathcal H_{i-1}$-modules corresponds, in the category of representations of $G_n$ 
generated by Iwahori-fixed vectors,  to the Jacquet functor with respect to the parabolic 
opposite to $P_i$.  (Note that this is not the same embedding as in Section \ref{s BZ afa}.)  Now there are two ways to construct the right adjoint of the restriction functor. One way is tensoring by $\mathcal H_n$ and the other, 
 by the second adjointness theorem of Bernstein, is the parabolic induction from $P_i$ to $G_n$. Hence, if $\sigma$ is a smooth 
representations of $M_i$, then, by the Yoneda lemma, we have a natural isomorphism of $\mathcal H_n$-modules 
\[ 
 \mathrm{Ind}_{M_{i}}^{G_n}(\sigma)^{I_n} \cong \mathcal H_n \otimes_{(\mathcal H_{n+1-i} \otimes  \mathcal H_{i-1})} (\sigma^{I_{M_i}}). 
 \] 
 

\begin{lemma} \label{lem realize} Let $P_{\mathrm{sgn}}^{i-1}=\mathcal H_{i-1} \otimes_{\mathcal H_{S_{i-1}}} \mathrm{sgn}$. 
The $\mathcal H_n$-module $(\mathrm{ind}_{Q_i}^{G_n}(\nu^{\frac{1}{2}}\pi^{(i)} \boxtimes \psi_{i-1}) )^{I_n}$ is isomorphic to 
\[  \mathcal H_n \otimes_{(\mathcal H_{n+1-i} \otimes  \mathcal H_{i-1})} ( (\nu^{\frac{1}{2}}\pi^{(i)})^{I_{n+1-i}} \boxtimes P_{\mathrm{sgn}}^{i-1} ).
\]
\end{lemma}

\begin{proof}
By the transitivity of inductions, since $G_n \supset P_i \supset Q_i$, 
\[ 
\mathrm{ind}_{Q_i}^{G_n}(\nu^{\frac{1}{2}}\pi^{(i)} \boxtimes \psi_{i-1}) \cong
 \mathrm{Ind}_{M_{i}}^{G_n} \left( \nu^{\frac{1}{2}}\pi^{(i)} \boxtimes \mathrm{ind}_{U_{i-1}}^{G_{i-1} } \psi_{i-1} \right).
\] 
 Lemma follows by taking Iwahori-fixed vectors and using Corollary \ref{cor realize iwahoir fix vector}. 
\end{proof}

Lemma \ref{lem realize} implies the following: 

\begin{corollary} \label{cor fintie gen}
Let $\pi$ be an irreducible generic representation of $GL(n+1,F)$. Then $\pi^{I_n}$ is a finitely generated $\mathcal H_n$-module.
\end{corollary}

\subsection{Locally nice representations}

We use the notations in Sections \ref{s BZ afa} and \ref{s bz gha}. This section does not directly use the realization of the Bernstein-Zelevinsky derivative via the Iwahori-Hecke algebras, but it is motivated by the Bernstein-Zelevinsky composition factors. The sign character plays a role in a number of places.

We first define a certain class of representations below. Since we only deal with Iwahori-fixed vector cases, it is more convenient to formulate the notions related to affine Hecke algebras. 
\begin{definition} \label{def relative generic}
Let $\pi_1$ be an irreducible generic representation of $GL(n+1,F)$. Let $\mathcal J$ be a maximal ideal of $\mathcal Z_n$.
 We say that $\pi_1$ is {\it locally nice} at $\mathcal J$ if the only irreducible representation $\pi_2$ of $GL(n,F)$ (with Iwahori-fixed vectors) satisfying the conditions that 
\begin{enumerate}
\item  $\mathrm{Hom}_{GL(n,F)}(\pi_1, \pi_2)\neq 0$ , and
\item $\pi_2^{I_n}$ is annihilated by $\mathcal J$,
\end{enumerate}
is the unique irreducible generic representation annihilated by $\mathcal J$.
\end{definition}

Examples for Definition \ref{def relative generic} are given below. Classifying locally nice representations is a $\mathrm{Hom}$-restriction problem.

\begin{example} \label{ex relatively generic} 
Let $\mathcal J$ be such that there exists only one isomorphism class of irreducible representations annihilated  by $\mathcal J$. This happens if the 
the irreducible generic representation of $GL(n,F)$ is also spherical, (see e.g. \cite{BM}, \cite{Re}). 
Then any generic representation of $GL(n+1,F)$ is locally nice at $\mathcal J$. 

\end{example}

We state some results useful in proving Theorem \ref{thm local structure}. 

\begin{theorem} \label{thm multi one} (see \cite{Pr, Pr2, AGRS})
Let $\pi_1$ be an irreducible generic representation of $GL(n+1,F)$ and let $\pi_2$ be an irreducible generic representation of $GL(n,F)$. Then
\[  \mathrm{Hom}_{GL(n,F)}(\pi_1, \pi_2) =1 
\]
\end{theorem}

\begin{lemma} \label{lem existence of whittaker space}
Let $\pi$ be an irreducible generic representation of $GL(n+1,F)$. Then $\pi|_{GL(n,F)}$ contains $\mathrm{ind}_{U_n}^{GL(n,F)} \psi_n$ as a submodule.
\end{lemma}
\begin{proof}
This follows from the Bernstein-Zelevinsky filtration (Theorem \ref{thm bz filtration}), definition of Bernstein-Zelevinsky derivatives (see Section \ref{ss ZD}) and the definition of a generic representation.
\end{proof}

Main ingredients of the proof of Theorem \ref{thm local structure} below are the multiplicity one theorem above (Theorem \ref{thm multi one}),  Definition \ref{def relative generic} and Corollary \ref{cor realize iwahoir fix vector}. 
We remark that Theorem \ref{thm local structure} is certainly not true without the condition of locally nicety.

\begin{theorem} \label{thm local structure}
Let $\pi$ be an irreducible generic representation of $GL(n+1,F)$ and let $I_n$ be the Iwahori subgroup of $GL(n,F)$. Regard $(\pi|_{GL(n,F)})^{I_n}$ as an $\mathcal H_n$-module. Let $\mathcal J$ be a maximal ideal in  $\mathcal Z_n$. 
 Let $\widehat{\mathcal Z}_n$ be the $\mathcal J$-adic completion of $\mathcal Z_n$. Set $\widehat{\mathcal H}_n =\widehat{\mathcal Z}_n \otimes_{\mathcal Z_n} \mathcal H_n$. Suppose $\pi$ is locally nice at $\mathcal J$ (see Definition \ref{def relative generic}). Then $\widehat{\mathcal Z}_n \otimes_{\mathcal Z_n} (\pi|_{GL(n,F)})^{I_n}$ is isomorphic to $\widehat{\mathcal H}_n \otimes_{\mathcal H_{S_n}} \mathrm{sgn}$ and hence is projective in the category of $\widehat{\mathcal H}_n$-modules.
\end{theorem}

\begin{proof}
For simplicity, set $\chi= (\pi|_{GL(n,F)})^{I_n}$, and let 
$\widehat{\chi}=\widehat{\mathcal Z}_n \otimes_{\mathcal Z_n} \chi$. Let $\widehat{\mathcal J}=\widehat{\mathcal Z}_n \otimes_{\mathcal Z_n} \mathcal J$. First of all, by Corollary \ref{cor fintie gen}, $\widehat{\chi}$ is a finitely generated $\widehat{\mathcal H}_n$-module. We divide the proof into several steps. 
\smallskip 

{\it Step 1}: 
Let $\widehat{\chi}'$ be the $\widehat{\mathcal H}_n$-submodule of $\widehat{\chi}$ generated by $\mathbf S_n( \widehat{\chi} )$, where $\mathbf S_n$ is the sign projector. 
 
\smallskip 
\noindent
{\it Claim}: $\widehat{\chi}'=\widehat{\chi}$.  
\smallskip 

\noindent
{\it Proof of the claim:} The key idea is to use  Definition \ref{def relative generic}.  
Let $\nu=\widehat{\chi}/\widehat{\chi}'$. Consider $\nu$ as a $\widehat{\mathcal Z}_n$-module. A quotient of a finitely generated module is finitely generated and furthermore $\widehat{\mathcal H}_n$ is finitely generated as $\widehat{\mathcal Z}_n$-module. Hence by the transitivity of finitely generatedness,  $\nu$ is 
 a finitely-generated $\widehat{\mathcal Z}_n$-module. Suppose $\nu \neq 0$. This implies  $\nu/\widehat{\mathcal J}\nu\neq 0$ (Nakayama's Lemma). 
Now $\nu/\widehat{\mathcal J}\nu$ descends to an $\widehat{\mathcal H}_n/\widehat{\mathcal J}\widehat{\mathcal H}_n$-module, which is finitely generated. 
Hence $\nu/\widehat{\mathcal J}\nu$ 
 is also finite-dimensional (and non-zero). Thus there exists a (non-zero) irreducible $\widehat{\mathcal H}_n$-quotient, say $\nu'$, 
 of $\nu/\widehat{\mathcal J}\nu$. However from our construction, $\nu'$ does not contain a sign representation and hence $\nu'$ is not generic (Corollary \ref{cor recover results}). This contradicts that $\mathrm{Hom}_{{\mathcal H}_n}({\chi}, \nu')=0$ by our assumption that $\pi$ is locally nice at $\mathcal J$. 
 \smallskip 

{\it Step 2} Since $\widehat{\chi}$ is finitely generated and $\widehat{\chi}'=\widehat{\chi}$ (from the proved claim), there exists a finite set of elements $x_1, \ldots, x_{r}$ in 
$\mathbf S_n( \widehat{\chi} )$ which generates $\widehat{\chi}$.  Assume that $r$ is the smallest  possible. 
 From our choices of generators $x_1, \ldots, x_r$, we have a surjective  map 
\[ \Psi: \bigoplus_{k=1}^r \widehat{\mathcal H}_n \otimes_{\mathcal H_{S_n}} \mathrm{sgn} \rightarrow \widehat{\chi}
\]
given by $(0,\ldots, 1 \otimes 1,  \ldots, 0) \mapsto x_k$, where $1 \otimes 1$ is in the $k$-th summand of $\bigoplus_{k=1}^r \widehat{\mathcal H}_n \otimes_{\mathcal H_{S_n}} \mathrm{sgn}$. 
Let 
\[ 
\mathscr{P}_l: \bigoplus_{k=1}^r \widehat{\mathcal H}_n \otimes_{\mathcal H_{S_n}} \mathrm{sgn} \rightarrow \widehat{\mathcal H} \otimes_{\mathcal H_{S_n}} \mathrm{sgn} 
\] 
 be the projection onto the $l$-th factor. The minimality of $r$ implies the following claim. 

\smallskip 
\noindent
{\it Claim}: $\mathscr{P}_l(\mathrm{ker} \Psi) \neq \widehat{\mathcal H}_n \otimes_{\mathcal H_{S_n}} \mathrm{sgn}$ for all $l$. 

\smallskip 
\noindent
{\it Claim}: $r \leq 1$.  

\smallskip 
\noindent
{\it Proof of the claim}:  Let 
\[ A_l:=\mathscr{P}_l\left(\mathrm{ker} \Psi + \bigoplus_{k=1}^r\widehat{\mathcal J}\left(\widehat{\mathcal H}_n \otimes_{\mathcal H_{S_n}} \mathrm{sgn} \right) \right) .\] 
If $A_l= \widehat{\mathcal H}_n\otimes_{\mathcal H_{S_n}} \mathrm{sgn}$ then $\mathscr{P}_l(\mathrm{ker} \Psi) = \widehat{\mathcal H}_n \otimes_{\mathcal H_{S_n}} \mathrm{sgn}$, 
by Nakayama's lemma, and this contradicts the previous claim. 
Thus $\widehat{\mathcal H}_n \otimes_{\mathcal H_{S_n}} \mathrm{sgn} /A_l$ is non-zero and moreover finite-dimensional. Let $\nu_l$ be an irreducible 
quotient.  By the Frobenius reciprocity, $\nu_l$ contains $\mathrm{sgn}$ and hence is the unique generic representation $\chi_{\mathrm{gen}}$ annihilated by $\mathcal J$. 
Hence this defines a map, denoted $f_l$, from $\widehat{\mathcal H}_n \otimes_{\mathcal H_{S_n}}\mathrm{sgn}$ to $\nu_l$ . Now we define a map $F_l: \bigoplus_{k=1}^r \widehat{\mathcal H}_n \otimes_{\mathcal H_{S_n}} \mathrm{sgn} \rightarrow \nu_l$ by $F_l=f_l \circ \mathscr{P}_l$. From our construction, $F_l(\mathrm{ker} \Psi)=0$ and hence descends to a map from $\chi$ to $\chi_{\mathrm{gen}}$.  
Note that $F_l$ are linearly independent, hence $\mathrm{Hom}_{\widehat{\mathcal H}_n}(\chi, \chi_{\mathrm{gen}}) \geq r$.  Theorem \ref{thm multi one}  proves the claim. 
\smallskip 

{\it Step 3:} We have shown that $\widehat{\chi}$ is isomorphic to $(\widehat{\mathcal H}_n \otimes_{\mathcal H_{S_n}}\mathrm{sgn}) /\mathrm{ker} \Psi$. It remains to prove $\mathrm{ker} \Psi = 0$. Suppose not. Let $a \otimes 1 \in \mathrm{ker} \Psi$ for some  non-zero $a \in \widehat{\mathcal Z}_n \otimes_{\mathcal Z_n} \mathcal A_n$. By Corollary \ref{cor realize iwahoir fix vector} and Lemma \ref{lem existence of whittaker space},  $\widehat{\mathcal H}_n \otimes_{\mathcal H_{S_n}} \mathrm{sgn}$ embeds into $(\widehat{\mathcal H}_n \otimes_{\mathcal H_{S_n}}\mathrm{sgn}) /\mathrm{ker} \Psi$, say the element $1 \otimes 1$ is mapped to an element 
represented by  $a' \otimes 1$ for some $a' \in \widehat{\mathcal Z}_n \otimes_{\mathcal Z_n} \mathcal A_n$. Now $a\otimes 1\neq 0$ is mapped 
to an element represented by $aa' \otimes 1$, but this one is in $\mathrm{ker} \Psi$. This is a contradiction. 
\end{proof}

Theorem \ref{thm local structure} provides a simple conceptual explanation to Conjecture \ref{conj prasad} for those locally nice representations. Our cases cover some that cannot be merely deduced from the composition factors of Bernstein-Zelevinsky filtrations and the Euler-Poincar\'e pairing. Moreover, as mentioned before, Theorem \ref{thm local structure} does not hold in general and thus a proof for a general $\mathrm{Ext}$-multiplicity result will require detailed understanding of structure or an alternate approach.

\begin{corollary} \label{cor ext branching}
Let $\pi_2$ be an irreducible generic representation of $GL(n,F)$ with Iwahori-fixed vectors annihilated by a maximal ideal $\mathcal J$ in $\mathcal Z_n$. 
Suppose $\pi_1$ is an irreducible generic representation of $GL(n+1,F)$  locally nice at  $\mathcal J$. 
  Then 
\[   \mathrm{Ext}^i_{GL(n,F)}(\pi_1, \pi_2) =0
\]
for all $i \geq 1$.
\end{corollary}
\begin{proof}  Corollary follows from Theorem \ref{thm local structure} using 
\[ 
\mathrm{Ext}^i_{\mathcal H_n}((\pi_1|_{GL(n,F)})^{I_n}, \pi_2^{I_n})\cong 
 \mathrm{Ext}^i_{\widehat{\mathcal H}_n}(\widehat{(\pi_1|_{GL(n,F)})^{I_n}}, \widehat{(\pi_2)^{I_n}}). \qedhere
 \]   
\end{proof}

\subsection{Branching rule for the Steinberg representation}

This section employs similar strategy as in Section \ref{s induced rep wc} to compute the $\mathcal H_n$-structure of the Steinberg representation of $GL(n+1)$. 
We work firstly with a general split reductive group $G$. 

Let $\mathrm{St}$ be the Steinberg representation of $G$. We use the notation from Section \ref{s induced rep wc}. In particular, $B$ is the Borel subgroup of $G$, 
$\bar U$ the unipotent radical of $\bar B$, the Borel opposite to $B$, and $X_w=B w\bar U$ are the Bruhat cells. Write $X=B\bar U$ for the open cell. 
 For any subset $J$ of simple roots $\Pi$, let $P_J$ be the standard parabolic subgroup associated to $J$ (and containing $B$).
In particular, $P_{\emptyset}=B$.  Let $C^{\infty}_c(P_J \setminus G)$ be the space of compactly supported smooth $P_J$-invariant functions on $G$. 
We use the following realization of the Steinberg representation: 
\[   \mathrm{St} = C^{\infty}_c(B\setminus G) / \sum_{ \emptyset  \neq J \subset \Pi} C^{\infty}_c(P_J \setminus G) .
\]
Thus we have a $\bar B$-equivariant map $\Omega: C_c^{\infty}(B\setminus X)\rightarrow\mathrm{St}$ given as the composition of natural maps 
\begin{align} \label{eqn cont steinb}
    C_c^{\infty}(B\setminus X) \rightarrow  C^{\infty}_c(B\setminus G) \rightarrow \mathrm{St}. 
\end{align}


\begin{proposition} \label{prop steinberg} 
The map $\Omega$ is a $\bar B$-equivariant isomorphism of $C_c^{\infty}(B\setminus X)$ and $\mathrm{St}$. 
\end{proposition}

\begin{proof}  Let $\mathbb C[W]$ denote the space of functions on $W$. Consider it a $W$-module for the action by right translations. For every 
simple root $\alpha$, let $W_{\alpha}=\{1,s_{\alpha}\}$. Then $\mathbb C[W_{\alpha}\backslash W]$  is a submodule of $\mathbb C[W]$ consisting of left 
$W_{\alpha}$-invariant functions. For injectivity we need the following lemma. 

\begin{lemma} Let $\delta\in \mathbb C[W]$ be the delta function corresponding to the identity element. Then $\delta$ cannot be written as a linear combination 
of elements in $\mathbb C[W_{\alpha}\backslash W]$ where $\alpha$ runs over all simple roots. 
\end{lemma}
\begin{proof} Functions in $\mathbb C[W_{\alpha}\backslash W]$ are perpendicular to the sign character. Hence any linear combination of such functions is also 
perpendicular to the sign character. But $\delta$ is not, hence lemma. 
\end{proof} 
We can now prove injectivity of $\Omega$. Let $f\in C_c^{\infty}(B\setminus X)$ be in the kernel of $\Omega$. 
Then there exist $f_{\alpha}\in C^{\infty}_c(P_{\alpha} \setminus G)$ such that $f=\sum_{\alpha\in\Pi} f_{\alpha}$.  For every $\bar u \in \bar U$, the function 
$w\mapsto f_{\alpha}(w\bar u)$ is in $\mathbb C[W_{\alpha}\backslash W]$. On the other hand, $w\mapsto f(w\bar u)$ is a multiple of $\delta$.   
Lemma implies that $f(\bar u)=0$. 

For surjectivity, let $V_r\subseteq  C_c^{\infty}(B \setminus G)$ be the subspace of functions supported on the union of the Bruhat cells $X_w$ for $w\in W$ such that $l(w) \leq r$. 
Let $V_w=C_c^{\infty}(B \setminus X_w)$. 
Then, if $r>1$, we have an exact sequence 
\[ 
0 \rightarrow V_{r-1} \rightarrow V_r \rightarrow \bigoplus_{l(w)=r} V_w \rightarrow 0. 
\] 
Let $v\in \mathrm{St}$  be the mage of $f\in V_r$. We need to show that $v$ is the image of some $f'\in V_{r-1}$. For every $w$ such that $l(w)=r$, pick $f_w\in V_r$ 
supported on $X_{w'}$ for $l(w) < r$ and $X_w$. Then $f-\sum_{l(w)=r} f_w \in V_{r-1}$. Since $r>1$, for every $w$ such that $l(w)=r$, there exists a simple root $\alpha$ 
such that $l(s_{\alpha} w) =r-1$.  
The group $G$ has a cell decomposition as a union of $Y_w=P_{\alpha} w \bar U$ where $w$ runs over all $w\in W$ such that $l(s_{\alpha} w) =l(w) -1$.  Note that 
$B\backslash X_w=P_{\alpha} \backslash Y_w$ for such $w$. Going back to our fixed $w$ such that $l(w)=r$, there exists a function
 $h_w \in C_c^{\infty}(P_{\alpha} \setminus G)$ such that the support of $h_w$ is on $Y_w$ and larger orbits, and $h_w=f_w$ on $B\backslash X_w=P_{\alpha} \backslash Y_w$. 
 The support of $h_w$, viewed as an element of $C_c^{\infty}(B \setminus G)$, is contained in $X_w$ and the union of $X_{w'}$ such that $l(w') < l(w)$. 
 Hence $f'=f-\sum_{l(w)=r} h_w \in V_{r-1}$ and $f'$ has the image $v$ in $\mathrm{St}$. Hence $\Omega$ is surjective.  
\end{proof}

Let $\mathrm{ch}_{I}$ be the characteristic function of $B(\bar U \cap I)$. 
Since $I=(B\cap I)(\bar U \cap I)$, it is an $I$-fixed element in $C_c^{\infty}(B \setminus G)$. Hence 
$v_0=\Omega(\mathrm{ch}_{I})$ spans the line of $I$-fixed vectors in $\mathrm{St}$. 

We now specialize to $GL(n)$.

 \begin{theorem} \label{thm global st} Let $\mathrm{St}_{n+1}$ be the Steinberg representation of $GL(n+1)$ and $v_0=\Omega(\mathrm{ch}_{I_{n+1}})$ the non-zero $I_{n+1}$-fixed vector.   The 
 $\mathcal H_n$-module $\mathrm{St}_{n+1}^{I_n}$ is generated by $v_0$ and isomorphic to $\mathcal H_n \otimes_{\mathcal H_{S_n}} \mathrm{sgn}$. In particular, it is projective. 
\end{theorem}
\begin{proof} Note that 
$\mathcal H_{S_n}$ acts on $v_0$ as the sign character. Let $D_n\cong (F^{\times})^n$ be the group of diagonal matrices and $B_n=D_n U_n$ be the Borel group of upper triangular matrices in $GL(n)$. 
Pick $\mathcal A_n \subseteq \mathcal H_n$, isomorphic to the group algebra of the lattice $D_n /(D_n \cap I_n)$, 
such that the Jacquet functor with respect to $\bar U_n$ corresponds to the restriction to $\mathcal A_n$. It suffices to show that 
$\mathrm{St}_{n+1}^{I_n}$ is freely generated by $v_0$ as an $\mathcal A_n$-module. This will be checked by passing to the Jacquet module with respect to $\bar U_n$. 
We have a decomposition $\bar U_{n+1} = \bar U_n \bar V_n$ where 
\[  \bar V_n = \left\{ \begin{pmatrix} I_{n \times n} & 0 \\ v & 1 \end{pmatrix}  : v=(a_1, \ldots a_n) \in M_{n \times 1}  \right\} \cong F^n .
\]
This identification and Proposition \ref{prop steinberg}, which says that $\mathrm{St_{n+1}}\cong C_c^{\infty} (\bar U_{n+1})$,  imply that there is an isomorphism of $D_n$-representations 
\[   
\Phi: (\mathrm{St_{n+1}})_{\bar U_n} \cong C_c^{\infty}(F^n). 
  \]
Furthermore,  $\Phi(v_0)$ is the characteristic function of $\mathcal O^n \subset F^n$. The theorem follows from the observation that $D_n /(D_n \cap I_n)$-translates of 
the characteristic function of $\mathcal O^n$ form a basis of $C_c^{\infty}(F^n)^{(D_n \cap I_n)}$. 
\end{proof}


\begin{corollary} \label{cor st locally nice}
The Steinberg representation $\mathrm{St}_{n+1}$ of $GL(n+1,F)$ is locally nice at every central character of $\mathcal H_n$. 
\end{corollary}

Note that Theorem \ref{thm local structure} for $\mathrm{St}_{n+1}$ can be recovered directly from Theorem \ref{thm global st}.

 \end{document}